\documentclass{amsart}

\usepackage{a4wide}

\usepackage[utf8]{inputenc}

\usepackage{bbm}
\usepackage{amsmath}
\usepackage{amssymb}
\usepackage{amsfonts}
\usepackage{amsthm}
\usepackage{hyperref}

\usepackage{todonotes}

\usepackage{color}

\usepackage[english]{babel}
\usepackage{graphicx}
\usepackage{indentfirst}

\theoremstyle{theorem}
\newtheorem{theorem}{Theorem}[section]
\newtheorem{proposition}[theorem]{Proposition}
\newtheorem{lemma}[theorem]{Lemma}
\newtheorem{corollary}[theorem]{Corollary}

\theoremstyle{definition}
\newtheorem{definition}[theorem]{Definition}
\newtheorem{remark}[theorem]{Remark}
\newtheorem{example}[theorem]{Example}

\newcommand{\be}{\mathbf{e}}

\newcommand{\D}{\mathcal{D}}

\newcommand{\Z}{\mathbb{Z}}
\newcommand{\N}{\mathbb{N}}
\newcommand{\R}{\mathbb{R}}
\newcommand{\Q}{\mathbb{Q}}
\newcommand{\K}{\mathbb{K}_{A}}
\newcommand{\E}{\mathcal{E}}
\newcommand{\bb}{B}
\newcommand{\F}{\mathcal{F}}
\newcommand{\qbb}{\Z^n((B))}
\newcommand{\zbb}{\Z^n[[B]]}
\newcommand{\qbbdual}{\Z^n((B^*))}

\def\mm\mathcal{ M }
\def\ll{\mathcal{L}}

\numberwithin{equation}{section}

\begin{document}
	\title[Rational self-affine tiles]{Rational self-affine tiles associated to standard and nonstandard digit systems}
%\markright{Tilings arising from digit systems with a rational matrix as a base}
	\author{Luc\'ia Rossi}
	\author{Wolfgang Steiner}
	\author{J\"org M. Thuswaldner}
	
	\address[L.R. \& J.M.T.]{Montanuniversit\"at Leoben, Franz Josef Straße~18, A-8700 Leoben, Austria}
	\address[W.S.]{Universit\'e Paris Cit\'e, CNRS, IRIF, F-75013, Paris, France}
	\date{\today}
	\thanks{The doctoral position of the first author is supported by the Austrian Science Fund (FWF) as part of the Discrete Mathematics Doctoral Program, project W1230. This work was supported by the project SYMDYNAR funded by the Agence Nationale de la Recherche and the Austrian Science Fund (ANR-23-CE40-0024-01, FWF I 6750), and bythe PHC Amadeus / \"OAD Amad\'ee project ``Topology, dynamics and number theory of fractal structures''. The second author was supported by the Agence Nationale de la Recherche through the project CODYS (ANR-18-CE40-0007).}
  
\email{lucia.rossi.moure@gmail.com}
\email{steiner@irif.fr}
\email{joerg.thuswaldner@unileoben.ac.at}

\maketitle

\begin{abstract}

	We consider digit systems $(A,\D)$, where $ A \in  \Q^{n\times n}$ is an expanding matrix and the \emph{digit set} $\D$ is a suitable subset of $\Q^n$.
	To such a system, we associate a self-affine set $\F = \F(A,\D)$ that lives in a certain representation space $\K$. If $A$ is an integer matrix, then $\K = \R^n$,
	while in the general rational case $\K$ contains an additional solenoidal
	factor. We give a criterion for $\F$ to have positive Haar measure,\textit{ i.e.},
	for being a \emph{rational self-affine tile}. We study topological
	properties of $\F$ and prove some tiling theorems. Our setting is very
	general in the sense that we allow $(A,\D)$ to be a \textit{nonstandard} digit system. A \textit{standard} digit system $(A,\D)$ is one in which we require $\D$ to be a complete system of residue class representatives w.r.t.\ a certain naturally chosen residue class ring. Our
	tools comprise the Frobenius normal form and character theory of locally
	compact abelian groups.
\end{abstract}

 \section{Introduction}
 This paper is a contribution to the theory of self-affine tiles whose foundations were established in the early 1990s by Bandt~\cite{MR1036982}, Kenyon~\cite{MR1185093}, Gr\"ochenig and Haas~\cite{MR1348740}, as well as Lagarias and Wang~\cite{MR1395075,MR1399601,MR1428817} and which has gained a lot of attention in the past decades.

We recall the definition of a self-affine tile. Let $A\in\R^{n\times n}$ be an expanding matrix ({\it i.e.}, all its eigenvalues lie outside the unit circle) with integer determinant, and let $\D\subset\R^n$ be a {\it digit set} with $|\D|=|\det A|$. Then we call the pair $(A,\D)$ a {\it digit system}. By Hutchinson~\cite{MR625600}, there exists a unique nonempty compact subset $\F=\F(A,\D)$ of $\R^n$ that satisfies the set equation 
\begin{equation}\label{eq:inttile}
A\F=\bigcup_{d\in\D}(\F+d). 
\end{equation}
If $\F$ has positive Lebesgue measure, it is called a \textit{self-affine tile}. Of special interest are the so-called integral self-affine tiles (see~\cite{MR1395075}), which are obtained when the matrix and the digits have integer coefficients.  By Bandt~\cite{MR1036982}, the Lebesgue measure of an integral self-affine tile $\F$ is certainly positive if $(A,\D)$ is a {\it standard digit system}, which means that $\D$ is a complete set of residue class representatives of  $\Z^n/A\Z^n$.  One very famous example of such an integral self-affine tile is Knuth's {\it twin dragon} (see~\cite[p.~206]{Knu}), whose boundary is a fractal set. Lagarias and Wang~\cite{MR1395075,MR1399601} also regarded the matter of $(A,\D)$ being {\it nonstandard}.  In this case, for a given matrix $A\in\Z^{n\times n}$ it is a highly nontrivial problem to characterize all digit sets $\D$ for which $\F(A,\D)$ has positive Lebesgue measure (cf.~\cite{AnLau19,LaiLau17}). 

{\it Rational self-affine tiles} are introduced by the second and third authors in~\cite{MR3391902}. They constitute a natural generalization of integral self-affine tiles to rational matrices which no longer need to have an integer determinant. In particular, a rational self-affine tile is defined in terms of an expanding matrix in $\Q^{n\times n}$ with irreducible characteristic polynomial and a digit set taken from a $\Z$-module defined in terms of this matrix. In the present paper we extend the theory of rational self-affine tiles by taking arbitrary expanding rational matrices (this includes the ``reducible case'', as is referred to in \cite{MR3391902}), and allowing nonstandard digit systems, as well as defining a representation space in a somewhat more general way. We provide results in the spirit of Lagarias and Wang~\cite{MR1395075,MR1399601} for this setting.

Let $A\in\Q^{n\times n}$ be an expanding matrix and let
\[
\Z^n[A] := \bigcup_{k=1}^{\infty} \left(\Z^n + A\Z^n + \dots + A^{k-1}\Z^n\right)
\]
be the smallest nontrivial $A$-invariant $\Z$-module containing $\Z^n$. We first define a {\em digit system} $(A,\D)$ where $A$ acts as a {\it base} and $\D \subset \Z^n[A]$ is some finite {\it digit set}, and explore its properties. In this digit system, the set $\Z^n[A]$
plays the role of $\Z^n$ in the sense that $(A,\D)$ can be used to represent elements of $\Z^n[A]$. We will always assume that the digit set $\D$ satisfies $|\D|=|\Z^n[A]/A\Z^n[A]|$, which turns out to be a natural size for a digit set (such digit systems are studied for instance in~\cite{thuswaldnerjankauskas2021}). After that, in order to set up the definition of our tiles, we introduce a representation space of the form $\K:=\R^n\times\Z^n((A^{-1}))$, where $\Z^n((A^{-1}))$ is a valuation ring of certain Laurent series of powers of $A^{-1}$ with coefficients in $\Z^n$. The ring $\Z^n((A^{-1}))$ is a solenoid, and in the one-dimensional case it is isomorphic to the ring of $b$-adic numbers for some $b\in\N$ (which is even a field when $b$ is a prime number). We establish a suitable ``diagonal'' embedding $\varphi$ that maps the elements of $\Z^n[A]$ into $\K$ in a natural way. This allows us to define the main object of study of this paper: the \textit{rational self-affine tile} $\F=\F(A,\D)\subset\K$, which arises as the unique nonempty compact solution of the set equation 
\[
A\F=\bigcup_{d\in\D}(\F+\varphi(d)),
\]
and can be interpreted as the set of ``fractional parts" of expansions in base $A$ with digits in $\D$ (embedded in $\K$). Since we want to study tilings of $\K$ induced by $\F$, we require rational self-affine tiles to have positive measure, which always holds when $\D$ is a complete  set of residue class representatives of $\Z^n[A]/A\Z^n[A]$, and in this case we call $(A,\D)$ a \textit{standard digit system}. However, we also allow $(A,\D)$ to be \textit{nonstandard}, and give a criterion in terms of the digits to guarantee positive measure of $\F$ in this general setting. We prove some topological properties of rational self-affine tiles, as well as the existence of two tilings given by translations of $\F$: the first one has a translation set defined in terms of the digits $\D$, and the second one is a multiple tiling where the translation set is a lattice obtained from the ring $\Z^n[A]$ by embedding it into the space $\K$. Before arriving to the proof of the existence of the multiple tiling, we present a careful analysis of the character group of the locally compact abelian group $\K$. The article \cite{rossi2021number} deals with a special one-dimensional case of what we present in this paper. It is accessible to a general audience and serves as an introduction to the topic. 

Our theory presents three main features that make its study difficult and rich. First of all, we treat the $n$-dimensional case: when dealing with matrices, computing quotients is not always so simple, and we make use of some machinery of linear algebra (like the Frobenius normal form) to solve some of these issues. Secondly, we deal with a space that has a $A^{-1}$-adic factor, and on it we define an ultrametric. 
More challenges arrive when we study the character group of the representation space. Finally, we consider nonstandard digit sets. This implies that sometimes we lose the group structure when considering certain digit expansions, which makes it harder to define tilings.

	  \section{Setting and definitions}
	  
In this section we introduce digit sets and present a way to compute their cardinality. After that we define the representation space, and finally arrive at the definition of a rational self-affine tile. 

What we are doing generalizes well-known facts on ``ordinary'' $\alpha$-ary expansions. In particular, let $\alpha$ be an integer with $|\alpha|>1$. Every real number has an expansion 
\[
\pm(d_md_{m-1}\ldots d_0.d_{-1}d_{-2}\ldots)_{\alpha} := \pm\sum_{j=-\infty}^m d_j\alpha^j
\] 
in base $\alpha$, where every $d_j$ is a {\it digit} taken from the set $\{0,1,\ldots,|\alpha|-1\}$ (the sign ``$\pm$'' is required only for positive~$\alpha$). We can associate to this digit system the set of \textit{integer expansions}, that is, the set of numbers that can be expressed using only nonnegative powers of the {\it base}~$\alpha$; it equals the lattice $\Z$.
Besides that there is the set of \textit{fractional parts}, namely, the numbers that can be expressed using only negative powers of the base~$\alpha$ (excluding the added $``-"$ sign at the beginning). This set is a compact interval given by $J=[0,1]$ for $\alpha>0$ and $J=\big[\tfrac{\alpha}{1-\alpha},\tfrac{1}{1-\alpha}\big]$ for $\alpha<0$.  The collection $\{J+z\,|\,z\in \Z\}$ forms a tiling of $\R$, and this property geometrically reflects the fact that almost every real has a unique $\alpha$-ary expansion. In our general setting, $\Z^n[A]$ will play the role of $\Z$, the rational self-affine tile $\F$ will play the role of $[0,1]$, and $\K$ will play the role of $\R$.

\subsection{Digit systems with rational matrices}\label{sec:DS}
 Many generalizations of $\alpha$-ary expansions (also known as \textit{radix expansions} or \textit{radix representations}) have been studied. Kemp\-ner~\cite{Kempner} and later R\'enyi \cite{Renyi:57} proposed expansions with respect to  nonintegral real bases. Knuth~\cite{Knuth:60} introduced complex bases and related them to fractal sets. Tilings of $\R^n$ arising form radix expansions were studied by Vince \cite{vince93}, and Kov\'acs \cite{MR2025177} considered digit systems in finite dimensional Euclidean spaces.
Akiyama et al.~\cite{MR2448050} introduced a type of expansion where the base is a rational number, and generalizations have been studied in \cite{BSSST:11} and \cite{SSTW}. In~\cite{MR3391902}, the authors considered digit expansions where the base is an algebraic number, which is equivalent to taking expansions w.r.t.\ rational matrices with irreducible characteristic polynomial. The setting from the present paper is a generalization of the one in \cite{MR3391902} for arbitrary expanding rational matrices.

Following Kov\'acs \cite{MR2025177}, given an integer matrix $A\in\Z^{n\times n}$, we can consider expansions in base $A$ with digits in a certain set $\D\subset\Z^n$, meaning we look at ways of expanding a vector $x\in\R^n$ in the form
\begin{equation}\label{radixexpansion}x=\sum_{j=-\infty}^k A^jd_j,\quad d_j\in\D,\end{equation}
and this is linked to the study of integral self-affine tiles that we have mentioned before (see~\cite{MR1395075}). We need the assumption that $A$ is expanding in order for this series to converge.
 
  If we consider a rational matrix $A\in\Q^{n\times n}$, then $\Z^n[A]$ is the natural generalization of $\Z^n$ (we refer to \cite{thuswaldnerjankauskas2021} for more on rational matrix digit systems). This leads to the following definition.

\begin{definition}[Digit system] Let $A\in\Q^{n\times n}$ be an expanding matrix and let $\D\subset \Z^n [A]$ be such that	
	$|\D|=|\Z^n[A]/A\Z^n[A]|.$ 
	Then we say that $(A,\D)$ constitutes a \textit{digit system}, where $A$ is the {\em base} and $\D$ is the {\it digit set}.
	When $\D$ is a complete set of residue class representatives of $\Z^n[A]/A\Z^n[A]$, we say that $(A,\D)$ is a {\it standard digit system} (following~\cite[p.~163]{MR1395075}). Otherwise, we say that $(A,\D)$ is a {\it nonstandard digit system}.
\end{definition}

In connection with digit systems, finite expansions are desirable.
For digit systems $(A, \D)$, the \emph{finiteness property}, stating that every vector $x \in \Z^n[A]$ has a finite expansion of the form 
$
x=A^kd_k+\dots+Ad_1+d_0
$ 
has been studied extensively (see~\cite{thuswaldnerjankauskas2021} and the references given there).
It is easily seen that the requirement that $\D$ is a complete system of residue class representatives of $\Z^n[A]/A\Z^n[A]$ is a necessary, but in general not sufficient, condition for $(A, \D)$ to have the finiteness property. Eventually periodic expansions have also been investigated.

\medskip

Computing the size of a digit set for a given expanding matrix $A\in\Q^{n\times n}$ amounts to computing the order of the quotient group $\Z^n[A]/A\Z^n[A]$, and this is not always straightforward. From here onwards, set 	
\begin{equation}\label{eq:a}
a:=|\Z^n[A]/A\Z^n[A]|, \quad b:=|\Z^n[A^{-1} ]/A^{-1}\Z^n[A^{-1}]|. 
\end{equation}
 We now show how to make use of the Frobenius normal form of $A$ to compute $a$ and $b$, and we prove that $|\det A|=\frac{a}{b}$, which will be crucial later.
Let $A\in\Q^{n\times n}$ with characteristic polynomial $\chi_A$ be given. Consider the space $\Q^n$ regarded as a finitely generated $\Q[t]$-module with the action of $t$ given by multiplication by $A$, that is, if $v\in\Q^n$ then $t\cdot v:=Av$, and the action can be linearly extended to all elements in $\Q[t]$. According to the structure theorem for finitely generated modules over principal ideal domains (see \cite[Chapter~12, Theorem~6]{DF:03}), there exists an isomorphism of the form
\[
\Q^n\simeq \bigoplus_{i=1}^k \Q[t]/(p_i),
\]
where $p_i \in \Q[t]$ are the so-called {\it invariant factors} of $\Q^n$, with the divisibility properties $p_1\mid p_2 \mid \ldots \mid p_k \mid \chi_A$. The polynomials $p_i$ are assumed to be monic, and with this assumption they are unique.
This implies that $A$ is similar to a block diagonal matrix $F=\mbox{diag}(C_1,\ldots. C_k)$, where $C_i$ is the companion matrix of $p_i$ ($1\leqslant i\leqslant k$). $F$ is the well-known {\em Frobenius normal from} of~$A$, also called {\it rational canonical form}, see \cite[Section~12.2]{DF:03}. Using this notation, we get the following result which shows how to compute the value of $a$.

\begin{proposition}\label{lem:frob}
	Let $A \in \Q^{n\times n}$ be given, let $p_i$ $(1\leqslant i \leqslant k)$ be the corresponding invariant factors, and consider the integer polynomials $q_i=c_ip_i \in \Z[t]$, where each $c_i\in\Z$ is chosen so that $q_i$ has coprime coefficients. Let $q_i^* \in \Z[t]$ be the reciprocal polynomial of $q_i$, namely $q_i^*(t):=t^{\deg{(q_i)}}q_i(t^{-1})$.
	Then
	\begin{equation}\label{computingaandb}
	a= \prod_{i=1}^k |q_i(0)|, \quad b= \prod_{i=1}^k |q_i^*(0)|,
	\end{equation}
	and $|\det A|=\frac ab$.
\end{proposition}

\begin{proof}
	Let $F=\mbox{diag}(C_1,\ldots, C_k)$ be the Frobenius normal form of $A$. Then it is clear that
	\begin{equation}\label{eq:arep}
	a=\prod_{i=1}^k |\Z^{\deg(q_i)}[C_i]/C_i\Z^{\deg(q_i)}[C_i]|.
	\end{equation}
	Let $C\in \Q^{m\times m}$ be the companion matrix of some polynomial $p\in\Q[t]$ and let $q=cp \in \Z[t]$, where $c\in\Z$ is chosen so that $q$ has coprime coefficients.  We claim that
\begin{equation}\label{eq:iso1}
 \Z^m[C] \simeq \Z[t] / (q).
\end{equation}

 To prove this, let $v \in \Z^m[C]$ be given; then it can be expressed as $v=\sum_{j=0}^{L}C^jv_j$ with $v_j\in\Z^m$, $L\geqslant 0$. For  $1\leqslant j\leqslant m$, denote by $\be_j\in\Q^m$ the $j$-th canonical basis vector. Clearly, each $v_j$ can be expressed in the canonical basis with integer coefficients. Because $C$ is a companion matrix, it follows that $C\be_j=\be_{j+1}$ for $1\leqslant j< m$, so all this yields that $v$ is of the form
	\begin{equation}\label{repofv}
	v= \sum_{j=0}^{\ell}b_jC^j\be_1,
	\end{equation}
	for some $\ell\in \N$ minimal and $b_0,\dots,b_\ell\in\Z$. If $\ell\geqslant m$, we will show that we can take $b_m,\ldots,b_\ell \in \{0,\ldots, |q^*(0)|-1\}$ (here, $q^* \in \Z[t]$ is the reciprocal polynomial of $q$, so $q^*(0)$ corresponds to the leading coefficient of~$q$). Indeed, note that $q(C)=0$ because $C$ is the companion matrix of $q$. Suppose that $b_\ell$ is not in $\{0,\ldots, |q^*(0)|-1\}$; then one can add or subtract $C^{\ell-m}q(C)=0$ in order to obtain another expression on the right side of (\ref{repofv}), without altering the value of $v$. This can be done the appropriate number of times, so we can assume w.l.o.g. that $b_\ell\in\{0,\ldots, |q^*(0)|-1\}$. Repeating this for $\ell-1,\ell-2,\dots,m$, one arrives at $b_m,\ldots,b_\ell \in \{0,\ldots, |q^*(0)|-1\}$, and the representation (\ref{repofv}) with this property and $\ell$ minimal is unique.
	
	 By \cite[Lemma~4.1]{SSTW} each polynomial $r\in\Z[t]/(q)$ can be expressed uniquely as $r= r'+\sum_{j=m}^\ell r_jt^j\mod q$, with $r'\in\Z[x]$, $\deg(r')<m$, $\ell\in\N$ and $r_m,\dots,r_\ell\in\{0,\dots,|q^*(0)|-1\}.$ Using this, one easily checks that 
	\[
	h: \Z^m[C] \to \Z[t] / (q); \quad \sum_{i=0}^{\ell}b_iC^i\be_1 \mapsto \sum_{i=0}^{\ell}b_it^i
	\]
	is an isomorphism and the claim in \eqref{eq:iso1} is proved.
	
	 Because $t\,h(G)=h(CG)$ for any $G\in\Z^m[C]$, this isomorphism implies that
	\[
	\Z^m[C]/C\Z^m[C] \simeq  \Z[t] / (q,t).
	\]
	It is easy to check that $| \Z[t] / (q,t)| = |q(0)|$ (see \cite[p.~1460]{SSTW}) and, hence, we have that \begin{equation}\label{indexmodC}| \Z^m[C]/C\Z^m[C] |=|q(0)|.\end{equation} Applying (\ref{indexmodC}) for $q=q_i$ $(1\leqslant i\leqslant k)$ in \eqref{eq:arep}, the left equation of (\ref{computingaandb}) follows. The right equation of (\ref{computingaandb}) is proved in the same way by replacing $A$ by $A^{-1}$. The assertion $\det A=\frac{a}{b}$ follows from (\ref{computingaandb}) (recall the definition of $q_i$ and the fact that each $p_i$ is monic), because
	\[
	|\det A| = \prod_{i=1}^k |\det C_i| =  \prod_{i=1}^k |p_i(0)| =  \prod_{i=1}^k \frac{|q_i(0)|}{|q_i^*(0)|}=\frac ab. \qedhere
	\]
\end{proof}

\subsection{The representation space}
	  
Properties of digit systems and digit expansions can be reflected geometrically via self-affine sets and tilings. The space $\K$, defined in what follows for a given $A\in\Q^{n\times n}$, will turn out to be a natural space where these sets and tilings can be defined. Suppose that we wanted to define a set $\F(A,\D)\subset\R^n$ satisfying \eqref{eq:inttile}. In order to define a tiling, we would require the union on the right side of \eqref{eq:inttile} to be essentially disjoint. We know that the action of $A$ scales the Lebesgue measure of a set by $|\det A|$, so if $A$ has a nonintegral determinant and $\F$ has positive measure, we would need $\D$ to have a nonintegral amount of digits, which is of course not doable. We will show that the action of $A$ in $\K$ multiplies the Haar measure of a set by $a$, where $a\in\N$ is as in~\eqref{eq:a}.

Standard digit systems $(A,\D)$ where the characteristic polynomial $\chi_A$ of $A\in\Q^{n\times n}$ is irreducible are considered in~\cite{MR3391902}. Rational self-affine tiles are introduced as subsets of a representation space of the form $\R^n\times\prod_{\mathfrak{p}}K_{\mathfrak{p}}$, where each $K_{\mathfrak{p}}$ is a completion of a number field $K$ (defined in terms of $\chi_A$) with respect to a certain absolute value $|\cdot|_{\mathfrak{p}}$. The representation space $\K$ from the present paper is a generalization of this, and can be defined in a simpler way. In the irreducible case, both settings are isomorphic.

\medskip
				
	 Let $A\in \Q^{n\times n}$ be expanding. For convenience of notation, from here onwards we set\[\bb:=A^{-1}. \]
	
	Consider the ring
	\[
	\Z^n(B)=\bigcup_{k\geqslant 1}(B^{-k}\Z^n+B^{-k+1}\Z^n+\dots+B^{k}\Z^n)
	\]
	(note that $\Z^n(A)=\Z^n(B)$). We define on $\Z^n(B)$ the \textit{$\bb$-adic valuation} $\nu:\Z^n(B)\rightarrow\Z\cup\{\infty\}$ as
	\begin{equation}\label{eq:nu}
	\nu(y):=
	\inf\{k\in\Z\;|\;y\notin B^{k+1}\Z^n[B]\}
	\end{equation}
	On $\Z^n(B)$ the \textit{$\bb$-adic metric} is defined by
	\begin{equation}\label{eq:dB}
	{\bf d}_\bb(y,y'):=b^{-\nu(y-y')},
	\end{equation}
	for $b$ as in \eqref{eq:a} and $y,y'\in\Z^n(B)$, with the convention that $b^{-\infty}=0$.
	
	\begin{definition}[$\bb$-adic series]\label{projlim} 
	We define the space $\qbb$ of {\it $\bb$-adic series} as the completion of $\Z^n(B)$ with respect to the metric ${\bf d}_\bb$. 	
	\end{definition}

We extend the metric ${\bf d}_\bb$ to the completion $\qbb$, and hence we extend the $\bb$-adic valuation $\nu$ to $\qbb$ so that it satisfies \eqref{eq:dB}.
Then, every nonzero $y\in\qbb$ can be expressed as Laurent series
\begin{equation}\label{eq:uniqueA^{-1} rep}
y=\sum_{j=\nu(y)}^{\infty}\bb^{j}y_j, \quad y_j\in\Z^n,
\end{equation}	
of powers of $\bb$ with coefficients in $\Z^n$, which converges w.r.t. the metric ${\bf d}_\bb$. Then $\nu(y)$ is the smallest index such that $y$ has an expansion \eqref{eq:uniqueA^{-1} rep} with $y_{\nu(y)}\neq 0$.
We denote by $\zbb$ the subring of $\qbb$ consisting of points $y\in\qbb$ with $\nu(y)\geqslant0$, the ring of power series in $B$ with coefficients in $\Z^n$.
The $\bb$-adic metric satisfies the \textit{ultrametric inequality}, namely
\[{\bf d}_\bb(y,y')\leqslant \max\{{\bf d}_\bb(y,y''),{\bf d}_\bb(y'',y') \}
\]
for every $y,y',y''\in\qbb$. This metric turns  $\qbb$ into a complete separable space, which is also a locally compact topological group. Thus there is a Haar measure $\mu_{\bb }$ on $\qbb$ which is normalized in a way that $\mu_\bb(\zbb)=1$, and we call it the \textit{$\bb$-adic measure}. 
If $ M \subset\qbb$ is a measurable set, then \begin{equation}\label{measureprop}\mu_{\bb }(A^k M )=b^k\mu_\bb( M ).\end{equation}
	\begin{remark}
		Let $n=1$. In this case, $A=\frac{a}{b}$ where $a$ and $b$ are coprime integers. Suppose $b\geqslant2$. 
		We obtain that $\qbb\simeq\Q_b$, where $\Q_b$ is the ring of $b$-adic numbers, and $\zbb\simeq\Z_b$, where $\Z_b$ is the ring of $b$-adic integers. \end{remark}

	\begin{definition}[The representation space]
		Given an expanding matrix $A\in\mathbb{Q}^{n\times n}$, define the {\it representation space} $\K$ as
		$$
		\K:=\R^n\times\qbb.
		$$

We endow the space $\K$ with the following structures:
\begin{enumerate}
	\item It inherits the structure of an additive group from its cartesian factors. 

\item Consider the group of matrices given by\[
\Z[A]:=\bigcup_{k\geqslant 1}( \Z A^{-k}+\Z A^{-k+1}+\dots+\Z A^{k}).\]
 Then $\Z[A]$ acts on $\K$ by multiplication, {\it i.e.}, $G\,(x,y)=(G x,G y)$ if $G\in\Z[A]$ and $(x,y)\in\K$.

\item We define the metric
$$\textbf{d}((x,y),(x',y')):=\mbox{max}\{\|x-x'\|,{\bf d}_\bb(y,y')\},$$
for $(x,y),(x',y')\in\K$, where $\|\cdot\|$ denotes the Euclidean norm in $\R^n$ and ${\bf d}_\bb$ is the $\bb$-adic metric in $\qbb$. This turns $ \K$ into a locally compact topological group. It is easy to check that, for every closed ball ${\bf B}_r(x,y)$ of radius $r> 0$ and center $(x,y)\in\K$, there is a decomposition
\[{\bf B}_r(x,y)={\bf B}_r(x)\times {\bf B}_r(y),
\]
into closed balls on each respective space. This characterizes the topology of $\K$.

\item We define a measure $\mu$ in $\K$ as the product measure
$$\mu:=\lambda\times\mu_{\bb },$$ 
with $\lambda$ being the Lebesgue measure in $\R^n$ and $\mu_\bb$ the $\bb$-adic measure in $\qbb$. Then $\mu$ is the Haar measure on $\K$ satisfying $\mu([0,1]\times\zbb)=1$.

\end{enumerate}	
	\end{definition}

\begin{remark} \label{assumptionforE}
	
	Note that, when $A\in\Z^{n\times n}$ is an integer matrix, the space $\qbb$ is trivial and plays no role, and hence $\K=\R^n$. However, $\K=\R^n$ may also happen in the noninteger case: for example, if $A=\begin{pmatrix}2&\frac{1}{2}\\0&3\end{pmatrix}$ we have $b=1$. Suppose $b=1$, {\it i.e.}, that $\Z^n[\bb]/\bb\Z^n[\bb]$ is trivial. Then by Lemma \ref{lem:frob} $\det A=a$ is an integer. The results presented in Section \ref{tiling theorems} are proven by Lagarias and Wang in \cite{MR1399601} for real expanding matrices with integer determinant. For this reason, in all that follows we assume $b \geqslant 2$.\end{remark}

\begin{lemma}\label{detofA}
	If $ M \subset\K$ is a measurable set, then $\mu(A M )=a\,\mu( M )$.

	\end{lemma}
	\begin{proof}
		Consider a measurable subset of $\K$ of the form $ M _1\times M _2$, where $ M _1\subset\R^n$ and $ M _2\subset\qbb$ are both measurable sets. We have seen in Proposition \ref{lem:frob} that $\det A=\frac{a}{b}$. Then
			\begin{equation*}
		\mu(A( M _1\times M _2))=\lambda(A M _1)\,\mu_{\bb }(A M _2)
		=\frac ab\,\lambda( M _1)\, b\,\mu_{\bb }( M _2)
		=a\,\mu( M _1\times M _2).
		\end{equation*}
		Since $\mu=\lambda\times\mu_\bb$ is a product measure, the $\sigma$-algebra of $\mu$-measurable sets is generated by sets of the form $ M _1\times M _2$. Therefore, if $ M \subset\K$ is measurable we have $\mu(A M )=a\,\mu( M ).$
	\end{proof}

 The previous lemma implies that, in some sense, the base $A$ has ``enough space" for $a$ digits when the digit system is embedded in $\K$. In fact, when considering a rational matrix $A$ with integer determinant, it suffices to take $|\D|=|\det A|$, because $A$ acts in some sense like an integer matrix: the action of $A$ scales the measure of a set by an integer factor. When $A$ has nonintegral determinant, it turns turns out that $A$ acts in $\K$ like an integer matrix, and that is why this space is appropriate for our purposes. This relation between the measure and the cardinality of the digit set will be important in all that follows.

\subsection{Rational self-affine tiles}
We proceed to introduce a set $\F$ associated to the digit system $(A,\D)$ that reflects features of its structure. It can be regarded as the set of ``fractional parts" of the digit system, in the sense that it plays the same role as the interval $[0,1]$ does for the decimal digit system. In Section~\ref{tiling theorems} we will study topological properties of $\F$. 

We need the following lemma, which is in the spirit of Lind~\cite{MR684244}.

\begin{lemma}\label{metricell}
Let $A\in\Q^{n\times n}$ be expanding and assume $b\geqslant 2$, with $b$ as in \eqref{eq:a}. Then there exists a metric $\boldsymbol{\ell}$ on $\K$ w.r.t.\ which the action of $\bb=A^{-1}$ is a contraction. In particular, there exists $0\leqslant \kappa<1$ such that
\begin{equation}\label{eq:elldef}
\boldsymbol{\ell}(\bb\, (x,y),\bb\, (x',y'))\leqslant \kappa\,\boldsymbol{\ell}((x,y),(x',y'))
\quad((x,y),(x',y')\in\K).
\end{equation}
\end{lemma}

\begin{proof}
	Let $\mathop{Spec}(A)$ denote the set of eigenvalues of $A$.  Since $A$ is expanding, there exists $\rho\in\R$ such that $1<\rho<\min\{|\eta|\;|\;\eta\in\mathop{Spec}(A)\}$. For $x\in\R^n$, define %the norm $\|\cdot\|'$ in $\R^n$ given by
\begin{equation}\label{eq:xprimedef}
\|x\|':=\sum_{k=0}^{\infty}\rho^k\, \|\bb^{k}x\|.
\end{equation}
	Since all the eigenvalues of $\rho\, \bb = \rho A^{-1}$ are strictly smaller than $1$ in modulus, the series on the right hand side of \eqref{eq:xprimedef} converges and $\|\cdot\|'$ becomes a norm in $\R^n$ that satisfies
	$$\|\bb \,x\|'=\frac{1}{\rho}\sum_{k=1}^{\infty}\rho^k\, \|\bb^{k}\,x\|\leqslant	\frac1\rho \, \|x\|'.$$
	Also, for all $y,y'\in\qbb$, it follows from the definition of the $\bb$-adic metric that ${\bf d}_\bb(By,By')=\frac{1}{b}\, {\bf d}_\bb(y,y').$ Let $(x,y),(x',y')\in\K$. Define on $\K$ the metric $\boldsymbol{\ell}$ given by
$$\boldsymbol{\ell}((x,y),(x',y')):=\mbox{max}\{\|x-x'\|',{\bf d}_\bb(y,y')\}.
$$
Then \eqref{eq:elldef} follows with $\kappa:=\max\{\frac1\rho,\frac{1}{b}\}<1$.
\end{proof}
 Note that ${\bf d}$ and $\boldsymbol{\ell}$ are equivalent, because $\|\cdot\|'$ is equivalent to $\|\cdot\|$.
  
We now introduce a suitable way to embed our digit system into the representation space $\K$. Define the {\it diagonal embedding} $\varphi$ as
\[
\varphi: \Z^n(B)\rightarrow\K;\quad
x\mapsto(x,x).
\] 	
	
\begin{definition}[Rational self-affine tile]
Let $(A,\D)$ be a digit system.
Define $\F=\F(A,\D)\subset\K$ as the unique nonempty compact set satisfying the set equation	
\begin{equation} \label{seteq}
A\F=\bigcup_{d\in\D}(\F+\varphi(d)).
\end{equation}
If $\mu(\F)>0$, then $\F$ is called a \textit{rational self-affine tile}.
\end{definition}

Because $A$ is expanding, Lemma~\ref{metricell} implies that the mapping \begin{equation*}
\K \to \K;\quad (x,y) \mapsto A^{-1}((x,y)+\varphi(d))
\end{equation*} 
is a contraction for each $d\in \D$. Let $\mathcal{H}(K)$ be the family of nonempty compact subsets of $\K$, and consider the map 
\begin{equation}\label{mappsi}
\varPsi: \mathcal{H}(K) \to \mathcal{H}(K); \quad X \mapsto \bigcup_{d\in\D}A^{-1} (X+\varphi(d)).
\end{equation} By Hutchinson~\cite{MR625600}, there is a unique nonempty compact set which is a fixed point of $\varPsi$, hence $\F$ is well defined. This set is the attractor of an iterated function system, which means that it is the Hausdorff limit of the sequence of compact sets $\{\varPsi^k(X)\}_{k\geq1}$, for any compact set $X$.

	 The set $\F$ can be interpreted as the set of ``fractional parts" of the digit system $(A,\D)$ embedded in $\K$; that is, every point of $\F$ can be expressed in base $A$ with digits in $\varphi(\D)$ using only negative powers of the base. In fact,  $\F$ is given explicitly by 
\begin{equation}\label{expliciteq}
\F=\Big\{ \sum_{j=1}^\infty A^{-j}\varphi(d_j)\;|\;d_j\in\D \Big\}.\
\end{equation}
Indeed, it is easy to see that $\F$ is nonempty, bounded, and satisfies \eqref{seteq}. The fact that $\F$ is closed follows by a Cantor diagonal argument.
%Let $(x,y)\in\K$ contained in the right side of (\ref{expliciteq}). Then
% \[
% x=\sum_{j=1}^\infty A^{-j}d_j\quad\mbox{and}\quad y=\sum_{j=1}^\infty \bb^{j}\,d_j,
% \]
%that is, both coordinates correspond to the same series, but the fir st one converges in $\R^n$ under the Euclidean metric, and the second one converges in $\qbb$ under the $\bb$-adic metric. It is easy to check that this set is nonempty and compact, and it satisfies the set equation (\ref{seteq}): indeed $$A\,(x,y)\in\F+\varphi(d_1)\subset\bigcup_{d\in\D}(\F+\varphi(d)).$$ The multiplication by $A$ is analogous to moving ``decimal point" one place to the right.
	%The existence and uniqueness of $\F$ is garanteed by Hutchinson's theorem.
	
Note that $\F$ is {\it self-affine} in the sense that it can be written as the union of $a=|\D|$ contracted affine copies of itself, because \eqref{seteq} is equivalent to $\F=A^{-1} (\F+\varphi(\D))$. When $\F$ has zero measure, it is a (generalization of a) Cantor set. 
	
	Suppose $\F$ has positive measure. In order to define a tiling, we want the union $\bigcup_{d\in\D}(\F+\varphi(d))$ to be essentially disjoint (that is, disjoint up to a $\mu$-measure zero set); since multiplication by $A$ on $\K$ enlarges the measure by a factor of $a$, then it is necessary for $\D$ to have exactly $a$ elements. We will show in the next section that, if $(A,\D)$ is a standard digit system, then $\F(A,\D)$ has positive measure. 
		%If $\F(A,\D)$ has positive masure but $\D$ is not a complete residue set modulo $A$, we say that $(A,\D)$ is a {\it nonstandard digit system}. 

 \begin{remark}\label{remarkwlogdigits}
 	 W.l.o.g.\ we will always assume that $0\in\D$. This can be done because replacing $\D$ by $\D-v$, where $v\in\Z^n[A]$ is a constant vector, means that $\F(A,\D-v)$ is a translation of $\F(A,\D)$, hence it is equivalent when we study the existence of tilings. 
% 	 [[[[Analogously, if $c\in\Z$ is a nonzero constant, it is easy to check using the explicit formula for the rational self-affine tile that $\F(A,c\D)=c\F(A,\D)$, so multiplying the digits by a constant only enlarges the tile. Using this, it is clear that we can choose an appropriate value for $c$ that allows us to assume w.l.o.g. that $\D\subset\Z^n$.]]]]
%\commJ{We need to be careful with multiplying. you distroy the ``standard'' property: look at ``primitive digit sets'' in Lagarias and Wang.}\commL{I fixed prop 4.6 so maybe we can omit the multiplication by a constant but we still can mention primitive digit sets}
 \end{remark}

	\section{Results on rational self-affine tiles}\label{tiling theorems}

	Let $(A,\D)$ be a digit system and let $\K$ be the representation space with metric $\textbf{d}$ and Haar measure $\mu$ as before. In this section, we give some equivalent topological and combinatorial conditions for the set $\F(A,\D)$ to have positive measure and we prove that, whenever $\F(A,\D)$ is a rational self-affine tile, it induces a tiling. We also study some topological properties and present an example.
	\subsection{Properties of rational self-affine tiles and a tiling theorem}
	We introduce some definitions before stating the results. To denote blocks of digits, let
\begin{equation}\label{defdk}
\D_k:=\big\{d_0 + Ad_1+\dots+A^{k-1}d_{k-1} \;|\; d_0,\ldots,d_{k-1}\in\D \big\}
\quad\hbox{and}
\quad
\D_\infty:=\bigcup_{k\geqslant 1}\D_k.
\end{equation}
From the set equation (\ref{seteq}), we deduce the {\it iterated set equation}
		\begin{equation}\label{iterset}
	A^k\F=\bigcup_{d\in\D_k}(\F+\varphi(d)).
	\end{equation}

	\begin{definition}[Uniform discreteness]
		We say that a set $M\subset\K$ is {\it uniformly discrete} if there exists $r>0$ such that every open ball of radius $r$ in $\K$ contains at most one point of $M$. 
	\end{definition}

		Our first result is a criterion for $\F(A,\D)$ to have positive measure formulated in terms of $\D$. It is an extension of \cite[Theorem~1.1]{MR1399601} and \cite[Theorem 10]{MR1185093} to the case of rational self-affine tiles.

	\begin{theorem}\label{uniformlydiscrete}
		Let $(A,\D)$ be a digit system, and let $\F=\F(A,\D)\subset\K$. Then
		$\F$ has positive measure if and only if for every $k\geqslant1$, all $a^k$ expansions in $\D_k$ are	distinct, and $\varphi(\D_\infty)$ is a uniformly discrete subset of $\K$.
	\end{theorem}
	
	\begin{proof}	%$(iv)\Rightarrow(i)$: 
	We omit some details of the proof, since the one in~\cite[p. 32 -- 34]{MR1399601} is similar, although it is provided in the setting of self-affine tiles in $\R^n$.
	
		Assume first that $\varphi(\D_\infty)$ is a uniformly
		discrete set and that all the elements in $\D_k$ are distinct for every $k\geqslant1$. %We want to	show that $\mu(\F)>0$. 
		%For $r>0$, define the closed ball,   
			%Let $\rho(A)$ denote the spectral radius of $A$, that is, the largest absolute value of its eigenvalues. 			
			Recall the metric $\boldsymbol{\ell}$ and the constant $0\leqslant\kappa<1$ defined in Lemma $\ref{metricell}$. % Let $\rho\in\R$ such that $1<\rho<\min\{|\eta_i|\}$, where $\{\eta_i\}$ is the set of eigenvalues of $A$.
			%for which multiplication by $\bb$ has contraction factor $0\leqslant\kappa<1$. %Let	$L:=\mbox{max}\{\frac{1}{\rho(A)},\frac1b \}\leqslant 1$, t
			Then $\boldsymbol{\ell}(\bb\,(x,y),0)\leqslant \kappa\,\boldsymbol{\ell}((x,y),0)$ for every $(x,y)\in\K$. Consider the closed ball ${\bf B}'_r(0):=\{(x,y)\in\K\;|\;\boldsymbol{\ell}((x,y),0)\leqslant r \},$ and let $(x,y)\in {\bf B}'_r(0)$. Let $\varPsi$ be the map defined in (\ref{mappsi}); then by Hutchinson~\cite{MR625600}, $\F$ is the Hausdorff limit of the sequence $\{\varPsi^k({\bf B}'_r(0)) \}_{k\geqslant 1}$.
		%Consider a digit and $d\in\D$; then	
%		\begin{equation} \label{ineq1}
%		\begin{split}
%		\boldsymbol{\ell}(\bb\, ((x,y)+\varphi(d)),0) & 
%	 \leqslant \kappa\,\boldsymbol{\ell}((x,y)+\varphi(d),0)\\
%		& \leqslant \kappa\,(r\boldsymbol{\ell}(\varphi(d),0))
%		\end{split}
%		\end{equation}	
		%If $\varPhi(X)=\bigcup_{d\in\D}A^{-1} (X+\varphi(d))$, 
		Suppose that $r\geqslant\frac{\kappa}{1-\kappa}\max_{d\in\D}\{\boldsymbol{\ell}(\varphi(d),0) \}$; then $\varPsi({\bf B}'_r(0))\subset{\bf B}'_r(0)$.	%This means that $\{\varPsi^k({\bf B}_r) \}_{k\geqslant 1}$ is a nested sequence that converges
		%to the attractor $\F$ of the iterated function system. %, h%ence $\F=\bigcap_{k\geqslant1}\varPsi^k({\bf B}_r).$
		Consequently, by Lebesgue's dominated convergence theorem, we get
	\[\mu(\F)=\lim_{k\to\infty}\mu(\varPsi^k({\bf B}'_r(0))).\] It suffices to find a set of positive measure contained in every $\varPsi^k({\bf B}'_r(0))$. Since $\varphi(\D_\infty)$ is uniformly discrete, there exists
		$\delta>0$ such that for every $(x,y)\neq(x',y')$ in $\varphi(\D_\infty)$ it
		holds that $\boldsymbol{\ell}((x,y),(x',y'))>\delta$. For $0<\varepsilon<\min\{\frac{\delta}{2},r \}$, consider the closed ball ${\bf B}'_\varepsilon(0)$.
		By hypothesis, $\D_k$ has $a^k$ distinct elements for all
		$k$, and all sets of the form ${\bf B}'_\varepsilon(0) +\varphi(d)$ for $d\in\D_k$ are pairwise disjoint. %also recall that		$\mu(A M )=a\mu( M )$ for every set $ M $ in $\K$.
		Therefore
		%\[a^k\mu(\varPsi^k({\bf B}_\varepsilon(0)))=\mu(A^k\varPsi^k({\bf B}_\varepsilon(0)))=\mu\Big(\bigcup_{d\in\D_k}{\bf B}_\varepsilon(0) +\varphi(d)\Big)=a^k\mu({\bf B}_\varepsilon(0)),\]
		$\mu(\varPsi^k({\bf B}'_\varepsilon(0)))=\mu({\bf B}'_\varepsilon(0))$	 for every $k$,
		%Since ${\bf B}_\varepsilon(0)\subset {\bf B}_r(0)$, we get $\varPsi^k({\bf B}_r(0))\supset\varPsi^k({\bf B}_\varepsilon(0))$  
		and hence		
		\begin{equation*}	
		\mu(\F) \geqslant \lim_{k\to\infty} \mu(\varPsi^k({\bf B}'_\varepsilon(0)))=\mu({\bf B}'_\varepsilon(0))>0.	
		\end{equation*}
		\indent For the converse, assume $\mu(\F)>0$. % Since for any measurable set $ M \subset\K$ we have $\mu(A M )=a\,\mu( M )$, then t
		%The iterated set equation \ref{iterset} implies 
		Then
		\begin{equation*} %\label{eq1}
		a^k\mu(\F)  =\mu(A^k\F) =\mu\Big(\bigcup_{d\in\D_k}(\F+\varphi(d))\Big)
		\leqslant	\sum_{d\in\D_k}\mu(\F+\varphi(d))\leqslant a^k\mu(\F),
		\end{equation*}
		so all the terms are equal. Hence, $|\D_k|=a^k$ and the union is essentially disjoint, i.e.,
		for $d\neq d'$ in $\D_k$, we get
		\begin{equation} \label{disjoint}
		\mu((\F+\varphi(d))\cap(\F+\varphi(d')))=0.
		\end{equation}

		It remains to show that $\varphi(\D_\infty)$ is uniformly discrete.	Suppose this was not the case, then we can find a sequence
		$\{(d_l,d'_l)\}_{l\geqslant 1}$ where, $d_l$ and
		$d'_l$ are distinct elements of some $\D_{k_l}$ for each $l\geqslant1$, and such that
		$$\lim_{l\to\infty}\textbf{d}(\varphi(d_l),\varphi(d'_l))=0.$$
		If $\mu(\F)>0$, then by  
		Federer~\cite[page~156, Corollary~2.9.9]{MR0257325},
		%\comment{Give a citation of this}
		there exists a Lebesgue point
		$(x^*,y^*)\in \F$.	Given $\varepsilon>0$, this implies the existence of a sufficiently small $r$ for which \begin{equation}\label{lebesgue0}
		\mu({\bf B}_r(x^*,y^*)\cap
		\F)\geqslant(1-\varepsilon)\,\mu({\bf B}_r(x^*,y^*)).
		\end{equation} Here, ${\bf B}_r(x^*,y^*)$ denotes the closed ball of center $(x^*,y^*)$ and radius $r$ w.r.t. the metric ${\bf d}$. Let $0< \varepsilon'<r$, and consider $(x,y)\in\K$ such that $\textbf{d}((x,y),0)<\varepsilon'<r$. Then
				\begin{equation}\label{lebesgue1}
		\begin{split}
		\mu({\bf B}_r(x^*,y^*)\cap (\F+(x,y))) &
		\geqslant\mu({\bf B}_{r-\varepsilon'}((x^*,y^*)+(x,y))\cap(\F+(x,y)))\\
		&\geqslant(1-\varepsilon)\mu({\bf B}_{r-\varepsilon'}(x^*,y^*)).
		\end{split}
		\end{equation}
		\indent Recall that $b$ satisfies \eqref{measureprop}. Note that ${\bf B}_r(y^*)={\bf B}_{b^{\lfloor \log_br \rfloor}}(y^*)$, and define $K:=\lfloor \log_b(r)\rfloor-\lfloor\log_b(r-\varepsilon') \rfloor>0$. Then $y\in{\bf B}_r(y^*)$ if and only if $\bb^Ky\in{\bf B}_{r-\varepsilon'}(\bb^Ky^*)$. Hence,
		\begin{equation*}
		\mu({\bf B}_{r-\varepsilon'}(x^*,y^*))=\lambda\Big(\frac{r-\varepsilon'}{r}{\bf B}_r(x^*)\Big)\,\mu_\bb(\bb^K{\bf B}_r(y^*))=\Big(\frac{r-\varepsilon'}{r}\Big)^nb^{-K}\mu({\bf B}_{r}(x^*,y^*)),
		\end{equation*}
		and thus for the appropriate value of $\varepsilon''>0$ it follows from \eqref{lebesgue1} that	\begin{equation}\label{epsilonprime}
		\mu({\bf B}_r(x^*,y^*)\cap (\F+(x,y)))>(1-\varepsilon'')	\mu({\bf B}_r(x^*,y^*)).
		\end{equation}		
		By inclusion-exclusion and combining (\ref{lebesgue0}) and (\ref{epsilonprime}), we get
		\begin{equation*}
		\begin{split}
		\mu((\F+(x,y))\cap \F) & >(1-\varepsilon-\varepsilon'')\mu({\bf B}_r(x^*,y^*)>0
		\end{split}
		\end{equation*}		
		for $(x,y)$ sufficiently close to 0. This implies that, for large enough $l$, 
		$$
		\mu((\F+\varphi(d_l))\cap(\F+\varphi(d_l')))= \mu(\F\cap(\F+\varphi(d_l'-d_l)))>0,
		$$ which is a contradiction.            
	\end{proof}

	\begin{corollary}\label{corollarystandard}
		If $(A,\D)$ is a standard digit system, then $\F(A,\D)$ has positive measure.
	\end{corollary}

The second result of this section gives some topological equivalences for $\F$ being a rational self-affine tile. It is in the spirit of \cite[Theorem~1.1]{MR1399601}.

	\begin{theorem}
		Let $(A,\D)$ be a digit system and let $\F=\F(A,\D)\subset\K$. 
		%the rational self-affine tile associated to this digit system. 
		The following assertions are equivalent:
		\begin{itemize}
			\item [$(i)$] $\F$ has positive measure.
			\item [$(ii)$] $\F$ has nonempty interior.
			\item [$(iii)$] $\F$ is the closure of its interior, and its boundary $\partial \F$ has measure zero.
%			\item [$(iv)$] For every $k\geqslant1$, all $a^k$ expansions in $\varphi(\D_k)$ are	distinct and $\varphi(\D_\infty)$ is a uniformly discrete set.
		\end{itemize}
	\end{theorem}
	
			\begin{proof}
			$(iii)\Rightarrow(ii)\Rightarrow(i)$ is trivial.
			
			\medskip
			
			$(ii)\Rightarrow(iii)$ is analogous to the proof found in \cite[p. 30 -- 32]{MR1399601}. 
%We show next that the boundary of $\F$ has measure zero. 
%			Suppose $\F$	contains an open set.  Then we can choose $k$ large enough so that there is an open ball in $A^k\F$ containing the set $\F+\varphi(d)$ for some $d$ in $\D_k$. Since we get an open ball covered
%			by compact sets, they must necessarily overlap. Therefore, the boundary of $\F+\varphi(d)$ is contained in $\bigcup_{d'\in\D_k\setminus\{d\}}\F+\varphi(d')$. We have shown in  (\ref{disjoint}) that for $d\neq d'$ in $\D_k$, the sets	$\F+\varphi(d)$ and $\F+\varphi(d')$ are essentially disjoint. 	This is equivalent to having disjoint interiors, so we get  $$\partial(\F+\varphi(d))\subset\bigcup_{d'\in\D_k\setminus\{d\}}\partial(\F+\varphi(d')).$$ Hence		
%			$$\mu(\partial(\F+\varphi(d)))\leqslant
%			\sum_{d'\in\D_k\setminus\{d\}}\mu(\partial(\F+\varphi(d))\cap\partial(\F+\varphi(d')))=0$$		
%			and therefore $\mu(\partial \F)=0$.
			\medskip

			$(i)\Rightarrow(ii):$ %Assume $\mu(\F)>0$. We want to show that $\F^\circ\neq\varnothing$.
		Because $\F$ has positive measure, it has a Lebesgue point $(x^*,y^*)$ satisfying (\ref{lebesgue0}). 
			This implies that we can consider a sequence $\varepsilon_k\searrow0$ together with a sequence of radii $r_k\searrow0$ such that, for every $l>0$,
			\begin{equation}\label{lebesguepoint}\mu(A^l({\bf B}_{r_k}(x^*,y^*)\cap \F))\geqslant (1-\varepsilon_k)\,\mu(A^l{\bf B}_{r_k}(x^*,y^*)).
			\end{equation}

\noindent{\bf Claim~1}: {\it For every index $k$ there exists a large enough $l_k>0$ and $(u^{(k)},v^{(k)})\in \K$  such that ${\bf B}_1(u^{(k)},v^{(k)})\subset A^{l_k}{\bf B}_{r_k}(x^*,y^*)$ with
				\[\mu({\bf B}_{1}(u^{(k)},v^{(k)})\cap A^{l_k}\F)\geqslant (1-C\,\varepsilon_k)\,\mu({\bf B}_{1}(u^{(k)},v^{(k)})),
				\]
where $(x^*,y^*)$ is a Lebesgue point of $\F$ and $C>0$ is a constant depending only on the space $\K$.}

To prove the claim, we draw on ideas from \cite[p.~35]{MR1399601}. With a slight abuse of notation, we will use ${\bf B}_r(\cdot)$ to denote closed balls of radius $r$ in each respective space. Fix $k$, and note that, since $A$ is expanding, $A^{l}{\bf B}_{r_k}(x^*)\subset \R^n$ is an ellipsoid whose shortest axis' length goes to infinity as $l$ goes to infinity. Consider $l_k>0$ large enough so that $b^{-l_k}\leqslant r_k$, and such that the ellipsoid $A^{l_k}{\bf B}_{r_k}(x^*)$ has a shortest axis greater than $4$. Define $E_{k}:=\{x\in A^{l_k}{\bf B}_{r_k}(x^*)\;|\;d(x,\partial( A^{l_k}{\bf B}_{r_k}(x^*)))\geqslant 1 \}$ as the set of points of $A^{l_k}{\bf B}_{r_k}(x^*)$ whose (Euclidean) distance from the boundary is at least $1$. Consider the set $2E_k-x^*$, obtained by doubling $E_k$ and centering around $x^*$. Then $E_k\subsetneq A^{l_k}{\bf B}_{r_k}(x^*)\subsetneq 2E_k-x^*$ and $\lambda(A^{l_k}{\bf B}_{r_k}(x^*))\leqslant \lambda(2E_k-x^*)= 2^n \lambda(E_{k}).$ %Note that $A^{l_k}{\bf B}_{r_k}(x^*)$ is strictly contained between $E_{k}$ and the ellipsoid obtained from $E_{k}$ by homothetically expanding it by multiplying by $2$ the length of every axis. 

Consider the compact subset of $\K$ given by $\mathcal{U}_k:=E_{k}\times A^{l_k} {\bf B}_{r_k}(y^*)$, which is trivially covered by the collection of unit balls $\mathcal{G}:=\{{\bf B}_1(u,v)\;|\;(u,v)\in\mathcal{U}_k \}.$ Let $\{{\bf B}_1(u_1,v_1),\ldots,{\bf B}_1(u_s,v_s) \}$ be a maximal disjoint subcollection of $\mathcal{G}$. Then $\mathcal{U}_k\subset \bigcup_{j=1}^{s}{\bf B}_2(u_j,v_j)$ because of the following reason: let $(x,y)\in\mathcal{U}_k$. If ${\bf B}_1(x,y)\notin\mathcal{G}$, then by maximality there exists $j\in\{1,\dots,s\}$ with ${\bf B}_1(x,y)\cap {\bf B}_1(u_j,v_j)\neq\varnothing$. Take $(x',y')\in{\bf B}_1(x,y)\cap {\bf B}_1(u_j,v_j)$. Then
				\[\textbf{d}((x,y),(u_j,v_j))\leqslant\textbf{d}((x,y),(x',y'))+\textbf{d}((x',y'),(u_j,v_j))\leqslant 2.
				\]
				This yields \begin{equation}\label{ballsineq}
				\begin{split} \mu(A^{l_k}{\bf B}_{r_k}(x^*,y^*))&=\lambda(A^{l_k}{\bf B}_{r_k}(x^*))\,\mu_\bb(A^{l_k}{\bf B}_{r_k}(y^*))\\
				& \leqslant 2^n\lambda(E_{k})\,\mu_\bb(A^{l_k}{\bf B}_{r_k}(y^*))\\
				&=2^n\mu(\mathcal{U}_k)\leqslant 2^n\sum_{j=1}^{s}\mu({\bf B}_2(u_j,v_j)).
				\end{split}
				\end{equation}
				
				Note that $\lambda({\bf B}_2(u_j))=2^n\lambda({\bf B}_1(u_j))$ while $\mu_\bb({\bf B}_2(v_j))=b^{\lfloor \log_b2 \rfloor}\mu_\bb({\bf B}_1(v_j));$
				hence
				\[\sum_{j=1}^{s}\mu({\bf B}_2(u_j,v_j))=2^{n}b^{\lfloor \log_b2 \rfloor}\mu\Big(\bigsqcup_{j=1}^{s}{\bf B}_1(u_j,v_j)\Big)
				\]
				because the unit balls are disjoint. Combining this with \eqref{ballsineq}, we have
				\begin{equation}\label{eq:ballest1}
				\mu(A^{l_k}{\bf B}_{r_k}(x^*,y^*))\leqslant C \, \mu\Big(\bigsqcup_{j=1}^{s}{\bf B}_1(u_j,v_j)\Big)
				\end{equation}
				for $C:=4^nb^{\lfloor \log_b2 \rfloor}$. We show next that all the balls ${\bf B}_1(u_j,v_j)={\bf B}_1(u_j)\times {\bf B}_1(v_j)$ are contained in $A^{l_k}{\bf B}_{r_k}(x^*,y^*)$. Fix $j\in\{1,\dots,s\}$. For the real part, $u_j\in E_{k}$, which means that it is a point of $A^{l_k}{\bf B}_{r_k}(x^*)$ which is at distance at least one from its boundary, hence ${\bf B}_1(u_j)\subset A^{l_k}{\bf B}_{r_k}(x^*)$. For the $\bb$-adic part, consider $y\in {\bf B}_1(v_j)$ and recall that $A^{-l_k}v_j\in {\bf B}_{r_k}(y^*)$. From the ultrametric inequality it follows
				\begin{equation}
				\begin{split}
				{\bf d}_\bb(A^{-l_k}y,y^*)&\leqslant \max\{{\bf d}_\bb(A^{-l_k}y,A^{-l_k}v_j),{\bf d}_\bb(A^{-l_k}v_j,y^*)\}\\
				&= \max\{b^{-l_k}{\bf d}_\bb(y,v_j),r_k \}\\
				&\leqslant\max\{b^{-l_k},r_k \}=r_k,
				\end{split}
				\end{equation}
				since we assumed $b^{-l_k}\leqslant r_k$. Therefore, $A^{-l_k}y\in {\bf B}_{r_k}(y^*)$ for every $y\in{\bf B}_1(v_j)$, and so ${\bf B}_1(v_j)\subset A^{l_k}{\bf B}_{r_k}(y^*)$. From \eqref{lebesguepoint} it follows that
				\begin{equation}\label{lebesguepoint2}\mu(A^{l_k}{\bf B}_{r_k}(x^*,y^*)\setminus (A^{l_k}{\bf B}_{r_k}(x^*,y^*)\cap A^{l_k}\F))\leqslant \varepsilon_k\mu(A^{l_k}{\bf B}_{r_k}(x^*,y^*)).
				\end{equation}
				Equations \eqref{eq:ballest1} and \eqref{lebesguepoint2} and the fact that $\bigsqcup_{j=1}^{s}{\bf B}_1(u_j,v_j) \subset A^{l_k}{\bf B}_{r_k}(x^*,y^*)$ imply
				\[\mu\Big(\bigsqcup_{j=1}^{s}{\bf B}_1(u_j,v_j)\setminus\Big(\bigsqcup_{j=1}^{s}{\bf B}_1(u_j,v_j)\cap A^{l_k}\F\Big)\Big)\leqslant \varepsilon_k\,C\,\mu\Big(\bigsqcup_{j=1}^{s}{\bf B}_1(u_j,v_j)\Big).
				\]
				Since the balls ${\bf B}_1(u_j,v_j)$ are pairwise disjoint and contained in $A^{l_k}{\bf B}_{r_k}(x^*,y^*)$, then for at least one $j_k\in\{1,\dots,s \}$ it holds that
				\[\mu({\bf B}_1(u_{j_k},v_{j_k})\cap A^{l_k}\F)\leqslant (1-\varepsilon_k\,C)\mu({\bf B}_1(u_{j_k},v_{j_k})),
				\]
				which yields Claim~1 with $(u^{(k)},v^{(k)})=(u_{j_k},v_{j_k})$. 
				
%				The fact that $\F$ has nonempty interior follows from this claim by an argument analogous to the one in \cite[p. 35, 36]{MR1399601}.
			Back to the main proof, Claim~1 together with the iterated set equation~\eqref{iterset} implies that, for every $k$, there exists $l_k>0$ and $(u^{(k)},v^{(k)})\in\K$ such that
			\begin{equation}\label{disjointineq1}
			\mu\Big({\bf B}_1(u^{(k)},v^{(k)})\cap \Big(\bigcup_{d\in\D_{l_k}}\F+\varphi(d)\Big) \Big)\leqslant (1-\varepsilon_k\,C)\mu({\bf B}_1(u^{(k)},v^{(k)})).
			\end{equation}
			%and hence
			%\[\mu({\bf B}_1(0)\cap (\bigcup_{d\in\D_{l_k}}\F+\varphi(d)-(u_{j_k},v_{j_k})) )\leqslant (1-\varepsilon_k\,C)\mu({\bf B}_1(0)).
			%\]
			Define the finite sets
			\[\mathcal{V}_k:=\{ \varphi(d)-(u^{(k)},v^{(k)})\;|\;d\in\D_{l_k},\;(\F+\varphi(d)-(u^{(k)},v^{(k)}))\cap \F\neq \varnothing \}.
			\]
			Then shifting the arguments inside the measures in \eqref{disjointineq1} by $-(u^{(k)},v^{(k)})$ and restricting to translates contained in $\mathcal{V}_k$ yields
			\[\mu({\bf B}_1(0)\cap (\F+\mathcal{V}_k ))\leqslant (1-\varepsilon_k\,C)\mu({\bf B}_1(0)).
			\]
			
			Note that $\F\cap {\bf B}_1(e)\neq\varnothing$ for every $e\in\mathcal{V}_k$. Thus, because $\F$ is bounded, all $\mathcal{V}_k\subset {\bf B}_R(0)$ for a sufficiently large constant $R$. Recall that $\varphi(\D_\infty)$ is a uniformly discrete set by Theorem
			\ref{uniformlydiscrete},
			and $\D_{l_k}\subset\D_\infty$,
			hence there exists $\delta>0$
			such that $\textbf{d}(e,e')\geqslant\delta$ 
			for every $e,e'\in\mathcal{V}_k$ for every $k$. This implies that the sequence of cardinalities $\{|\mathcal{V}_k|\}_{k\geqslant 1}$ is bounded.  Therefore, $\{\mathcal{V}_k\}_{k\geqslant1}$ has a convergent subsequence $\{\mathcal{V}_{k_j}\}_{j\geqslant 1}$ whose limit, denoted by $\mathcal{V}$, is a finite set. Then
			\begin{equation}
			\begin{split}
			\mu({\bf B}_1(0)\cap (\F+\mathcal{V}))&\geqslant \liminf_{j\to\infty}\mu({\bf B}_1(0)\cap (\F+\mathcal{V}_{k_j} ))\\
			&\geqslant \liminf_{j\to\infty} (1-C\,\varepsilon_{k_j})\mu({\bf B}_1(0))=\mu({\bf B}_1(0)).
			\end{split}
			\end{equation}
			Because $T$ is closed this implies that $(\F+\mathcal{V})\cap {\bf B}_1(0)={\bf B}_1(0)$. Thus $\F+\mathcal{V}$ is a finite union of translates of the compact set $\F$ containing inner points. Baire's theorem implies that $\F$ has nonempty interior.		\end{proof}
	
In what follows, we will restrict ourselves to the case where $\F$ has positive measure. We have referred to $\F$ in this case as a tile, because we will show that there exists a \textit{tiling} of the space $\K$ by translates of $\F$. 

\begin{definition}[Tiling, self-replicating tiling, multiple tiling]\label{deftiling}
	Assume $\mu(\F)>0$. Let $\mathcal{S}\subset\K$ and consider the collection $\{\F+s\;|\;s\in\mathcal{S}\}$, which we denote as $\F+\mathcal{S}$ with a slight abuse of notation. 
	\begin{enumerate}
		\item $\F+\mathcal{S}$ is said to be a \textit{tiling} of $\K$ if it is a covering of $\K$ such that, for any $s\neq s'$ in $\mathcal{S}$, it holds that $\mu((\F+s)\cap(\F+s'))=0$, or, equivalently, if $\F+s$ and $\F+s'$ have disjoint interiors. $\mathcal{S}$ is called a tiling set for $\F$. We say that $\F+\mathcal{S}$ tiles $\K$.
		\item 	 $\F+\mathcal{S}$ is said to be a \textit{self-replicating tiling} if there exists an expanding linear map $T$ on $\K$ such that,
		for each $s\in\mathcal{S}$, there exists a finite subset $J(s)\subset\mathcal{S}$ with
		$$T(\F+s)=\bigcup_{s'\in J(s) }(\F+s').$$
		\item 	$\F+\mathcal{S}$ is said to be a \textit{multiple tiling} of $\K$, if there exists $k\ge 1$ such that $\mu$-almost every point of $\K$ is contained in exactly $k$ distinct sets of the form $\F+s$ with $s\in\mathcal{S}$.

	\end{enumerate}

\end{definition}

	It follows that a self-replicating tiling is completely determined by the set of tiles that touch the origin 0. We call a self-replicating tiling atomic if the origin touches exactly one tile.
	
	%For any set $\chi$ in $\K$ we associate the difference set
	%$$\varDelta(\chi):=\chi-\chi=\{x-x'\;|\;x,x'\in\chi \}$$
	%We define the difference radix expansion set
	%$$\varDelta(A,\D):=\bigcup_{k=1}^{\infty}\varDelta(\D_k)$$
	%It is clear that, for all $y$ in $\K$,
	%$$\varDelta(A,\D+y)=\varDelta(A,\D)$$
	%and if $0\in\D$ then $$\varDelta(A,\D)=\varDelta(\D_{\infty                             })$$
	
	We want to study the nature of the tilings of $\K$ obtained using self-affine tiles. For any $k\geqslant 1$, consider the difference sets
	
	$$\D_k-\D_k=\{d-d'\;|\;d,d'\in\D_k \},$$
	
	%Consider the set $$\varphi(\varDelta\D_{\infty}):=\{d-d'\;|\;d,d'\in\D_{\infty} \}$$

	and define $$\Delta:=\bigcup_{k=1}^{\infty}\varphi(\D_k-\D_k).$$

	\begin{theorem}
		Suppose that $\F$ contains an open set. Then:
		
		\begin{enumerate}
			\item [$(i)$] There exists a set of translations $\mathcal{S}\subset\Delta$ such that $\F+\mathcal{S}$ tiles $\K$. Furthermore, there exists a translate $\mathcal{S}'$ of $\mathcal{S}$ such that $\F+\mathcal{S}'$ is an atomic self-replicating tiling of $\K$, and the expanding linear map associated to it is of the form $T=A^k$ for some sufficiently large $k$.
			\item [$(ii)$] If $\Delta$ is a group, then $\F+\Delta$ is a tiling.
		\end{enumerate}
	\end{theorem}

\begin{proof}
	$(i)$ The proof is analogous to the one in \cite[Theorem~1.2]{MR1399601}.

				\medskip
	$(ii)$ Assume that $\Delta$ is a group. Since by $(i)$ there is a subset $\mathcal{S}\subset \Delta$ for which $\F+\mathcal{S}$ tiles $\K$, $\F+\Delta$ is a covering of $\K$. Given any $s\neq s'$ in $\Delta$, then it suffices to show $$0=\mu((\F+s)\cap(\F+s'))=\mu(\F\cap(\F+s'-s)).$$	Because $\Delta$ is a group, then $s'-s\in\Delta$. That means that there exists $k\geqslant 1$ and $d\neq d'$ in $\D_k$ such that $s'-s=\varphi(d-d')$, so the assertion is equivalent to $$\mu((\F+\varphi(d))\cap(\F+\varphi(d')))=0,$$which holds by (\ref{disjoint}).	
\end{proof}

	\subsection{Example}
We give now an example of a standard and a nonstandard digit system, together with an illustration of a rational self-affine tile.

\begin{example}

	Let 
	\[
	A = \begin{pmatrix}
	2 & 1\\
	0 & \frac53
	\end{pmatrix}\in\Q^{2\times2}.
	\]

	Since its characteristic polynomial is $\chi_A(x)=(x-2)(x-\tfrac53)$, $A$ is expanding. 		We proceed to find the representation space $\K$. %, so we need to compute the groups $	\E_k:=\Z^2[\bb ]/\bb^{k}\Z^2[\bb ].$
	We have
	\[
	\bb=A^{-1} =\begin{pmatrix}
	\frac12 & -\frac{3}{10}\\
	0 & \frac35
	\end{pmatrix},
	\]
	and we get
	\[\Z^2[\bb ]=\left\lbrace \begin{pmatrix}
	\frac{s}{2^n5^m}\\
	\frac{t}{5^l}
	\end{pmatrix}\;|\;n,m,l\in\N,\;s,t\in\Z\right\rbrace \]
%	An element of $\bb \Z^2[\bb ]$ is of the form
%	
%	\[
%	\begin{pmatrix}
%	\frac12 & -\frac{3}{10}\\
%	0 & \frac35
%	\end{pmatrix}\begin{pmatrix}
%	\frac{s}{2^n5^m}\\
%	\frac{t}{5^l}
%	\end{pmatrix}=\begin{pmatrix}
%	\frac{s}{2^{n+1}5^m}-\frac{3t}{5\cdot 2^{l+1}}\\
%	\frac{3s}{5^{l+1}}
%	\end{pmatrix}.
%	\]
%	This yields
	and \[
	\bb \Z^2[\bb ]=\left\{\begin{pmatrix}
	\frac{s}{2^n5^m}\\
	\frac{3t}{5^l}
	\end{pmatrix}\;|\;l,m,n\in\N,\;s,t\in\Z
	\right\}.
	\]
	Then, a residue set  $\E$ for the quotient $\Z^n[B]/B\Z^n[B]$ is given by
	$$\E=\left\lbrace \begin{pmatrix}
	0\\
	0
	\end{pmatrix},\begin{pmatrix}
	0\\
	1
	\end{pmatrix},\begin{pmatrix}
	0\\
	2
	\end{pmatrix} \right\rbrace, $$
	hence $b=3$.
%
%	Computation yields that
%	$$\bb^k=\begin{pmatrix}
%	(\frac12)^k & d_k\\
%	0 & (\frac35)^k
%	\end{pmatrix}$$ with $d_k=-\frac{3}{10}\sum_{i+j=k-1,\;i,j\geqslant 0}\left( \frac12\right) ^i\left( \frac35\right) ^j.$ This entails
%	\[\bb^k\E=\left\lbrace \begin{pmatrix}
%	0\\
%	0
%	\end{pmatrix},\begin{pmatrix}
%	0\\
%	\left( \frac35\right) ^k
%	\end{pmatrix},\begin{pmatrix}
%	0\\
%	2\left( \frac35\right) ^k
%	\end{pmatrix} \right\rbrace.
%	\]
	%It is not hard to check that, for this example, it is possible to establish an isomorphism $\qbb\simeq \Q_3$, and thus we regard the elements of $\K$ as points in $\R^2\times\Q_3$. 
	We will consider a digit set $\widetilde{\D}$ and in Figure~\ref{tile2}, we will illustrate a rational self-affine tile associated to $(A,\widetilde{\D})$, and represent it in $\R^3$ by embedding $\R^2\times \qbb$ in $\R^3$.

	We will first find a standard digit system $(A,\D)$, and for that we compute the quotient $\Z^2[A]/A\Z^2[A]$. %An arbitrary element of $\Z^2[A]$ is of the form
	%		\[
	%		\begin{pmatrix}
	%		a_0\\
	%		b_0
	%		\end{pmatrix} + A\begin{pmatrix}
	%		a_1\\
	%		b_1
	%		\end{pmatrix}+ A^2\begin{pmatrix}
	%		a_2\\
	%		{\bf B}_2
	%		\end{pmatrix}+\cdots+ A^n\begin{pmatrix}
	%		a_n\\
	%		b_n
	%		\end{pmatrix}
	%		\]
	%		
	%		with $n\in\N,\,a_j,b_j\in\Z$. 
	Powers of $A$ are of the form

	\[
	A^k = \begin{pmatrix}
	2^k & c_k\\
	0 & (\frac53)^k
	\end{pmatrix}\in\Q^{2\times2},
	\]	where $c_k=\sum_{i,j:\,i+j=k-1}{2^i\left( \frac53\right) ^j}.$ We have	\[
	\Z^2[A]=\left\{\begin{pmatrix}
	\frac{s}{3^n}\\
	\frac{t}{3^m}
	\end{pmatrix}\;|\;n,m\in\N,\;s,t\in\Z 
	\right\},
	\]
	hence, an element of $A	\Z^2[A]$ is of the form

	\[
	\begin{pmatrix}
	2 & 1\\
	0 & \frac53
	\end{pmatrix}\begin{pmatrix}
	\frac{s}{3^n}\\
	\frac{t}{3^m}
	\end{pmatrix}=\begin{pmatrix}
	\frac{2s}{3^n}+\frac{t}{3^m}\\
	\frac{5t}{3^{m+1}}
	\end{pmatrix}.
	\]
	
	Note that $\frac{2s}{3^n}+\frac{t}{3^m}\equiv\frac{5t}{3^{m+1}}\mod2
	$		in $\Q$, because if we multiply both sides by $3^{n+m +1}$ we get $3^{m+1}2s+3^{n+1}t\equiv t \mod2$, and $3^{n}5t\equiv t\mod 2$. This yields
	\[
	A\Z^2[A]= \left\{\begin{pmatrix}
	\frac{s}{3^n}\\
	\frac{5t}{3^m}
	\end{pmatrix}\;|\;n,m\in\N,\;s,t\in\Z,\;s\equiv t\mod2
	\right\}.
	\]
	
	A complete set of residue class representatives of $\Z^2[A]/A\Z^2[A]$ is given by $$
	\D=\left\lbrace \begin{pmatrix}
	0\\
	0
	\end{pmatrix},\begin{pmatrix}
	0\\
	1
	\end{pmatrix},\begin{pmatrix}
	0\\
	2
	\end{pmatrix},\begin{pmatrix}
	0\\
	3
	\end{pmatrix},\begin{pmatrix}
	0\\
	9
	\end{pmatrix},\begin{pmatrix}
	1\\
	0
	\end{pmatrix},\begin{pmatrix}
	1\\
	1
	\end{pmatrix},\begin{pmatrix}
	1\\
	2
	\end{pmatrix},\begin{pmatrix}
	1\\
	3
	\end{pmatrix},\begin{pmatrix}
	1\\
	9
	\end{pmatrix}		
	\right\rbrace ,
	$$		and so $(A,\D)$ is a standard digit system. Note that  $a=10$ and so $|\det A|=\frac{10}{3}=\frac{a}{b}$, as expected.

	Next, we want to find a nonstandard digit system. Note that we can write $\mathcal{D}=\mathcal{R}_1+\mathcal{R}_2$, where \[\mathcal{R}_1=\left\lbrace \begin{pmatrix}
	0\\
	0
	\end{pmatrix},\begin{pmatrix}
	0\\
	1
	\end{pmatrix},\begin{pmatrix}
	0\\
	2
	\end{pmatrix},\begin{pmatrix}
	0\\
	3
	\end{pmatrix},\begin{pmatrix}
	0\\
	9
	\end{pmatrix}
	\right\rbrace, \quad\mathcal{R}_2=\left\lbrace\begin{pmatrix}
	0\\
	0
	\end{pmatrix},\begin{pmatrix}
	1\\
	0
	\end{pmatrix}
	\right\rbrace ,
	\] and the decomposition of the digits as a sum is unique.
	Consider 
	\[\widetilde{\D}:=\mathcal{R}_1+A\mathcal{R}_2=\left\lbrace \begin{pmatrix}
	0\\
	0
	\end{pmatrix},\begin{pmatrix}
	0\\
	1
	\end{pmatrix},\begin{pmatrix}
	0\\
	2
	\end{pmatrix},\begin{pmatrix}
	0\\
	3
	\end{pmatrix},\begin{pmatrix}
	0\\
	9
	\end{pmatrix},\begin{pmatrix}
	2\\
	0
	\end{pmatrix},\begin{pmatrix}
	2\\
	1
	\end{pmatrix},\begin{pmatrix}
	2\\
	2
	\end{pmatrix},\begin{pmatrix}
	2\\
	3
	\end{pmatrix},\begin{pmatrix}
	2\\
	9
	\end{pmatrix}		
	\right\rbrace,\] which is not a residue set for $\Z^n[A] \mod A$. We show that $\F(A,\widetilde{\D})$ has positive measure. 		\begin{figure}[h]
		\centering\includegraphics*[width=0.6\linewidth]{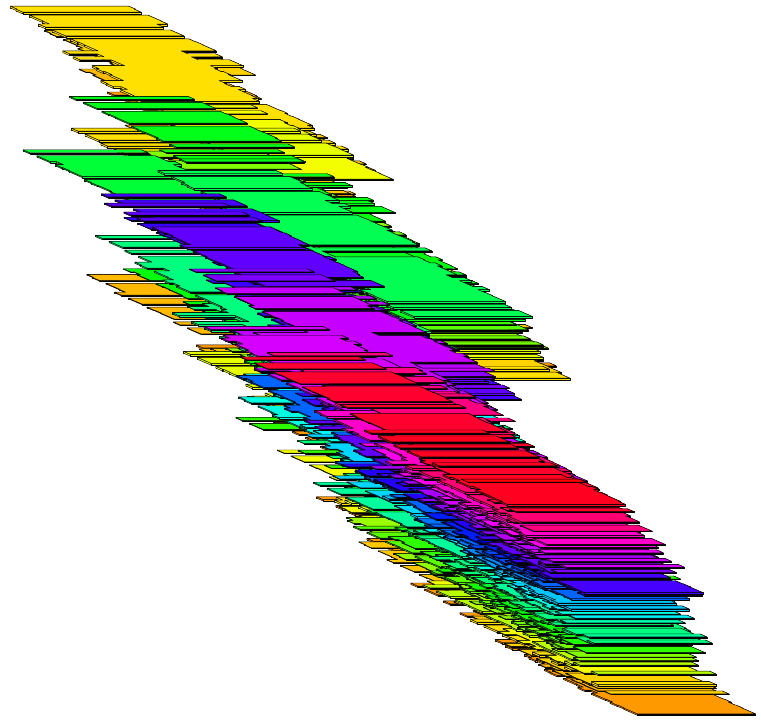}
		\caption{The tile $\mathcal{F}(A,\widetilde{\D}).$}
		\label{tile2}
	\end{figure}
	Let \[E_1:=\Big\{\sum_{j=1}^{\infty}A^{-j}\varphi(e_{j})\;|\;e_j\in\mathcal{R}_1\Big\}
\quad\mbox{ and }\quad
	E_2:=\Big\{\sum_{j=1}^{\infty}A^{-j}\varphi(e'_{j})\;|\;e'_j\in\mathcal{R}_2\Big\}.
	\]
	Now, the unique decomposition of the elements of $\widetilde{\D}$ yields that 
	\[\begin{split}
	\F(A,\widetilde{\D})&=E_1+A\,E_2=E_1+\Big\{\sum_{j=0}^{\infty}A^{-j}\varphi(e'_{j+1})\;|\;e'_j\in\mathcal{R}_2\Big\}\\
	&=E_1+E_2+\varphi(\mathcal{R}_2)=\F(A,\D)+\varphi(\mathcal{R}_2).\\
	\end{split}
	\]
	
	Since $(A,\D)$ is a standard digit system, the set $\F(A,\D)$ has positive measure (see Corollary~\ref{corollarystandard}), and hence so does $\F(A,\widetilde{\D})$, which means that it is a rational self-affine tile associated to a nonstandard digit system. We illustrate it in Figure \ref{tile2}. Recall that the representation space is not Euclidean, so we embedded the points of $\K$ into $\R^3$ in order to draw the picture. We did that so the figure still reflects some of the properties of the set, but it is not a completely faithful representation, since this is not possible due to the $\bb$-adic factor.

\end{example}

\section{Characters and multiple tiling}

The main result of this section states that, whenever $\F(A,\D)$ is a tile, it gives a multiple tiling of $\K$. Before arriving to the proof, we introduce some definitions in order to find the Pontryagin dual of $\K$, and we give a complete description of its characters. We make use of the character theory of locally compact abelian groups to show that the multiplication by $A$ is ergodic on a certain torus in $\K$, and use this to prove the existence of the multiple tiling.

\subsection{Some basic results and definitions}

	The module $\Z^n[A]$ plays a principal role in the study of tilings by rational self-affine tiles. We will prove first that, embedded into the representation space $\K$, this module becomes a lattice, and we show later that it is a translation set for a multiple tiling given by copies of $\F$. First, we formalize the notion of lattice in our setting. 
	
\begin{definition}[Lattice]
		A subset $\Lambda$ of $\K$ is a {\em lattice} in $\K$ if it satisfies the three following conditions:
	\begin{enumerate}
		\item $\Lambda$ is a group.
		\item $\Lambda$ is uniformly discrete, i.e., there exists $r>0$ such that every open ball of radius $r$ in $\K$ contains at most one point of $\Lambda$.
		\item $\Lambda$ is relatively dense, i.e., there exists $R>0$ such that every closed ball of radius $R$ in $\K$ contains at least one point of $\Lambda$.
	\end{enumerate}
\end{definition}

We show next that $\varphi(\Z^n[A])$ satisfies these properties. We state a lemma first.
		
	\begin{lemma}\label{fin}
		 %Let $a=|\Z^n[A] / A\Z^n[A]|$ and $b=|\Z^n[A^{-1} ]/A^{-1} \Z^n[A^{-1} ]|$. Then 
	%	$$\det(A)=\frac{a}{b}.$$
		There exists an integer $K \geqslant 1$ such that
		\begin{equation}			\label{indexK}
		\Z^n \cap \bb \Z^n[\bb] = \Z^n \cap \left(\bb\Z^n + \bb^{2}\Z^n+\dots+\bb^{K}\Z^n\right).
		\end{equation}
%		and
%		\[
%		\Z^n \cap A^{-1}  \Z^n[A] = \Z^n \cap \left(A^{-1} \Z^n + A^{-2}\Z^n+\dots+A^{-K}\Z^n\right)
%		\]
	\end{lemma}
	
	\begin{proof}
%			Using the second isomorphism theorem we get 
%		
%		
%		
%		
%		\[
%		\Z^n[A] / A\Z^n[A] \cong \Z^n/(A\Z^n[A] \cap \Z^n).
%		\]
		%We prove first the following claim: 
	For $k \geqslant 1$,  define the lattices (in $\R^n$)
			\[
			\ll_k[\bb] := \sum_{j=1}^{k}\bb^{j}\Z^n \subset \R^n.
			\]
			
			Since $\ll_k[\bb]$ contains $\bb\Z^n$ and $B$ is invertible, the lattice $\ll_k[\bb] \subset \R^n$ has full rank. Consider a nonzero integer $m_k$ such that  $m_k\ll_k[\bb] \subset \Z^n$. 
			Then $m_k\ll_k[\bb]$ has finite index in~$\Z^n$. 
			From this fact, one deduces that the intersection  $ \Z^n \cap \ll_k[\bb]$
			has finite index in~$\Z^n$ for every $k\geqslant1$. 
			Therefore, the chain of nested  lattices
			\[
			\left(\Z^n \cap \ll_1[\bb]\right) \subset \left(\Z^n \cap \ll_2[\bb] \right)\subset \dots \subset \left(\Z^n \cap \bb \Z^n[\bb]\right) \subset \Z^n
			\]
%			has the property
%			\[
%			\bigcup_{k=1}^{\infty} \Z^n \cap \ll_k[\bb] = \Z^n \cap \bb\Z^n[\bb],
%			\]
			must eventually stabilize after some $K\geqslant1$. 
%Doing the same for $A^{-1}$, we find some $K_2$. The claim follows by taking $K$ large enough. It yields
%	
%	\[a= |\Z^n/\Big(\Z^n\cap\sum_{j=1}^{K}A^j\Z^n\Big)|=x^{(K)}.
%	\]
%	
%	
%	In the same way we get 
%	\[
%	b = |\Z^n/\Big(\Z^n\cap\sum_{j=1}^{K}A^{-j}\Z^n\Big)|=y^{(K)}
%	\]
%	and the corollary follows from Lemma~\ref{detaprox}.
\end{proof}

	\begin{proposition}\label{lattice}
		The set $\varphi(\Z^n[A])$ is a lattice in~$\K$.
	\end{proposition}
	
	\begin{proof}
		The fact that $\varphi(\Z^n[A])$ is a group follows from the additive group structure of $\Z^n[A]$ because $\varphi$ is a group homomorphism.
		
		To prove the uniform discreteness of $\varphi(\Z^n[A])$, we claim that there exists $0<r\leqslant1$ such that $\textbf{d}(\varphi(z),0) \geqslant r$ for every nonzero $z \in \Z^n[A]$. If $\textbf{d}(\varphi(z),0) \geqslant 1$ then we are done. 
		Suppose on the contrary that $\textbf{d}(\varphi(z),0) < 1$.  Since $z \in \Z^n[A]$, we can write it as
		\begin{equation}\label{eq_expr3}
		z = \sum_{j=0}^k A^j z_j, \quad z_j \in \Z^n,
		\end{equation} 
		with $z_k\neq0$, and there is a minimal index $k$ with this property. 
		If $z_k \notin \bb \Z^n[\bb ]$, then $\bb^{k}z \in \Z^n[\bb] \setminus \bb \Z^n[\bb]$, so one has $\textbf{d}_\bb(z,0)=b^k\textbf{d}_\bb(\bb^{k}z,0) \geqslant 1$, in contradiction to $\textbf{d}(\varphi(z),0) < 1$.  
		Thus, $z_k \in \Z^n\cap \bb \Z^n[\bb]$. By Lemma \ref{fin}, there are vectors $w_1$, $w_2$, $\dots$, $w_{K} \in \Z^n$ such that
		\[
		z_k = \bb w_1+\bb^{2}w_2+\dots+\bb^{K}w_{K}.
		\]
		We claim that, in such case, one must have $k \leq K$. Suppose that $k > K$. Then
		\[
		z =  A^k z_k + \sum_{j=0}^{k-1} A^j z_j=  \sum_{j=0}^{k-K-1} A^j z_j + \sum_{j=k-K}^{k-1}(z_j + w_{k-j})A^j =  \sum_{j=0}^{k-1} A^j z_j', \quad z_j' \in \Z^n.
		\]
		However, that would contradict the minimality of $k$. Therefore, $k \leq K$. Now, let $m$ denote the least common multiple of the denominators of the entries of $A$. Since $z$ is a sum of integer vectors multiplied by $A^k$, $0 \leqslant k \leqslant K$, the non-zero entries of $z$ are at least $1/m^{K}$ in absolute value. Hence, $\textbf{d}(\varphi(z),0) \geqslant \|z\|\geqslant 1/m^{K} = r$.

		We now turn to the proof of relative denseness. 
		Let $(x, y) \in \K$ be arbitrary. Choose $x' \in \Z^n$ to be the closest integer vector to $x$, so that $\|x - x'\| < 1$, and choose $y'\in \Z^n[A]$ such that ${\bf d}_\bb(y,y') \leqslant 1$ (this holds by taking $y'=\{y\}_\bb$, see Definition \ref{integerfracpart} below). Hence, ${\bf d}((x,y),(x',y'))\leqslant1$. Choose $y''\in\Z^n$ to be the closest integer vector to $x'-y'\in\R^n$. Let $z:=y'+y''\in\Z^n[A]$. Then $\|x'-z\|=\|x'-y'-y''\|<1$. Moreover, ${\bf d}_\bb(z,y')={\bf d}_\bb(y'',0)\leqslant 1$ because $y''\in\Z^n$. Therefore, ${\bf d}((x',y'),\varphi(z))\leqslant1.$ 		
%		Also, there exists  $y'' \in \Z^n$, such that the vector $z = y''+y' = y''+\sum_{j=0}^k A^j d_j' \in \Z^n[A]$ satisfies $\|x' - z\| < 1$. Moreover, $y'' \in \Z^n$ implies $\|y''\|_{\bb }=\|z-y'\|_{\bb } \leqslant 1$. 
%		From this, $\|x - z\| < 2$ and $\|y-z\|_{\bb } < 2$. 
		This yields $\textbf{d}((x, y),\varphi(z)) \leqslant 2$, which proves the Lemma.
	\end{proof}

The next step is to define a space $\qbbdual$ that will be crucial later when we study the characters of $\K$. Prior to that, we prove the following Lemma.

\begin{lemma}\label{lattice2}
	The group $\Z^n[A]\cap\Z^n[\bb]$ is a lattice in $\R^n$.%\commL{check if it can be shortened} 
\end{lemma}
\begin{proof}

	We show that $\Z^n[A]\cap\Z^n[\bb]\subset\Z^n+A\Z^n+\dots+A^K\Z^n$ for some $K\geqslant 1$.	Let $z \in \Z^n[A] \cap \Z^n[\bb]$. 
	Since $z \in \Z^n[A]$, write $z = \sum_{j=0}^k A^j z_j$, with $z_j \in \Z^n$, $z_k \neq 0$ and $k$ minimal. If $k =0$, then $z \in \Z^n$. Assume $k \geqslant 1$. 
	Since $z \in \Z^n[\bb]$, from solving for $z_k$ it follows that $z_k \in \Z^n \cap \bb\Z^n[\bb]$. 
By Lemma \ref{fin}, one can find vectors $w_1, w_2, \dots, w_{K} \in \Z^n$ such that
%	By replacing $A$ for $\bb$ in the proof of Corollary $\ref{fin}$ it holds that there exists an integer $K \geqslant 1$, that depends only on $A$, such thatB one finds that
%\[
%z_k = A^{-k}z - \sum_{j=0}^{k-1} A^{j-k} z_j.
%\]
%	
%	\[
%	\Z^n \cap A^{-1}  \Z^n[A^{-1}] = \Z^n \cap \left(A^{-1} \Z^n + A^{-2}\Z^n+\dots+A^{-K}\Z^n\right)
%	\]

	\begin{equation}\label{eq_expr2}
	z_k = \bb w_1+\bb^{2}w_2+\dots+\bb^{K}w_{K},
	\end{equation}
	and by proceeding like in the proof of Proposition \ref{lattice}, one shows $k \leqslant K$.
	%We showed in the proof of Lemma \ref{lattice} that $k \leqslant K$. 
	Therefore the inclusion follows. This implies that $\Z^n[A]\cap\Z^n[\bb]$  is contained in a lattice (in $\R^n$), and also it trivially contains the lattice $\Z^n$. Since it is a group, it is itself a lattice.
\end{proof}
For an arbitrary lattice $\Lambda\subset \R^n$, define its\textit{ dual lattice }by \[
\Lambda^*:=\{x\in\Z^n\;|\;\langle x,z\rangle \in\Z\mbox{ for every }z\in\Lambda \},
\]
where $\langle \cdot,\cdot\rangle$ denotes the usual scalar product in $\R^n$. Denote by ${A^*}$ the transpose of the matrix $A$, and let $\Lambda$ and $\Gamma$ be full rank lattices in $\R^n$ such that
\begin{equation}\label{eq:duallattice}
\Z^n[A]\cap\Z^n[\bb]\subset\Lambda \quad\mbox{ and }\quad \Gamma[{A^*}]\cap\Gamma[{\bb^*}]\subset\Lambda^*\subset\Z^n.
\end{equation}
This is possible because $\Z^n[A]\cap\Z^n[\bb]$ is a lattice in $\R^n$ by Lemma \ref{lattice2}, and the proof that $\Gamma[{A^*}]\cap\Gamma[{\bb^*}]$ is also a lattice is analogous. 

\begin{definition}[$\bb$-adic 
and $\bb^*$-adic 
expansions]
	
Let $\E\subset\Z^n$ be a complete set of residue classes of $\Z^n[\bb ]/\bb\Z^n[\bb]$ with $0\in\E$.
%We have defined $\E\subset\Z^n$ to be a complete set of coset representatives of $\Z^n[\bb]/\bb\Z^n[\bb]$. This yields a unique $\bb$-adic expansion. However, the same can be done for arbitrary lattices. We now define the $\bb^*$-adic expansion in the following way
 Then every $y\in \qbb$ has a unique expansion of the form	\begin{equation}\label{badicexp}y=\sum_{j=\nu(y)}^{\infty} {\bb}^{j}y_j,\quad y_j\in\E,
\end{equation}which we call the \textit{$\bb$-adic expansion} of $y$ with coefficients in $\E$. Recall that $\nu(0)=\infty$, so the $\bb$-adic expansion of $0$ is the empty sum.

Let $\bb^*$ denote the transpose of $\bb$. Consider the full rank integer lattice $\Gamma \subset\R^n$ satisfying (\ref{eq:duallattice}) and let $\E^*\subset\Gamma\subset\Z^n$ be a complete set of coset representatives of $\Gamma[\bb^*]/{\bb^*}\Gamma[\bb^*]$ with $0\in\E^*$. Consider the space $\qbbdual$ defined analogously to $\qbb$. Then every $s\in \qbbdual$ has a unique expansion of the form
\begin{equation}\label{bstaradicexp}s=\sum_{j=\nu^*(s)}^{\infty} {\bb^*}^{j}s_j,\quad s_j\in\E^*,
	\end{equation}
	where $\nu^*$ is the valuation in $\qbbdual$ defined in the same way as $\nu$. We call this the \textit{$\bb^*$-adic expansion} of $s$ with coefficients in $\E^*$.
\end{definition}

\begin{definition}[$\bb$-adic and $\bb^*$-adic fractional and integer part]\label{integerfracpart}
Given $y\in\qbb$ with $\bb$-adic expansion \eqref{badicexp}, we define the \textit{$\bb$-adic fractional part} and the \textit{$\bb$-adic integer part} of $y$, respectively, as
\[ \{ y \}_{\bb}:=\sum_{j=\nu(y)}^{-1} {\bb}^{j}y_j,\quad \lfloor y \rfloor_{\bb}:=\sum_{j=0}^{\infty} {\bb}^{j}y_j.
\]
%\begin{definition}[${\bb^*}$-adic series]
%	Let $\bb^*:={{A^*}^{-1}}$. Consider a lattice $\Gamma$ satisfying (\ref{eq:duallattice}) and let $\E^*\subset\Gamma$ be a complete set of coset representatives of $\Gamma[\bb^*]/{\bb^*}\Gamma[\bb^*]$ with $0\in\E^*$. Denote by $\qbbdual$ the ring of formal Laurent series of powers of $\bb^*$ with coefficients in $\E^*$.  Then every $y\in \qbbdual$ is of the form
%	\[y=\sum_{j=\nu^*(y)}^{\infty} {\bb^*}^{j}y_j,\quad y_j\in\E^*,
%	\]
%	where $\nu^*(y)\in\Z$ is defined as the smallest index such that $y_{\nu^*(y)}\neq 0$ whenever $y\neq 0$. Define $\zbbdual$ as the subring of $\qbbdual$ given by the points $y$ such that $\nu^*(y)\geqslant0$, together with $y=0$.
%
%\end{definition}
%

Given $s\in\qbbdual$ with $\bb^*$-adic expansion \eqref{bstaradicexp}, we define the \textit{$\bb^*$-adic fractional part} and the \textit{$\bb^*$-adic integer part} of $s$, respectively, as
\[ \{ s \}^*_{\bb}:=\sum_{j=\nu^*(s)}^{-1} {\bb^*}^{j}s_j,\quad \lfloor s \rfloor^*_{\bb}:=\sum_{j=0}^{\infty} {\bb^*}^{j}s_j.
\]
 \end{definition}

From here onwards, whenever we have a $\bb$-adic series (resp.\ $\bb^*$-adic series), we assume the coefficients to lie in $\E$ (resp. $\E^*$).

\begin{remark} Recall that $b=|\E|$. We claim that, if $b=1$, then the multiple tiling theorem holds. Note that, in this case, $\det A=a$ is an integer. %We have mentioned in Remark \ref{assumptionforE} that the results of Section \ref{tiling theorems}
	Indeed, an analogous version of Theorem \ref{multipletilingthm} is proven by Lagarias and Wang in \cite{MR1395075} for integer matrices. However, they show in \cite[Lemma~2.1]{MR1395075} that this results also holds for self-affine tiles associated to expanding real matrices $A\in\R^{n\times n}$ with integer determinant, %as long as $\F(A,\D)$ is a \textit{lattice-self affine tile}, that is, if it has positive measure and 
	as long as there exists an $A$-invariant lattice in $\R^n$ containing the difference set $\D-\D$. If $\F(A,\D)$ is a rational self-affine tile and $b=1$, then  $\Z^n[A]\cap\Z^n[\bb]$ is a lattice in $\R^n$ by Lemma \ref{lattice2}, and it is $A$-invariant because $\Z^n[\bb]/\bb\Z^n[\bb]$ is trivial and hence $A\Z^n[\bb]\subset \Z^n[\bb]$. Consider $c\in\Z\setminus\{0\}$ such that $c\D\subset\Z^n$; then $\D-\D\subset\frac{1}{c}(\Z^n[A]\cap\Z^n[\bb])$, which is an $A$-invariant lattice. %Also, we can assume w.l.o.g. that it contains the difference set $\D-\D$, because we can take $\D\subset \Z^n\subset\Z^n[A]\cap\Z^n[A^{-1}]$, see Remark \ref{remarkwlogdigits}.
	Therefore, the assumption that $b\geqslant 2$ made in Remark \ref{assumptionforE} also applies to this section.
\end{remark}

\begin{remark}\label{spanlatticeremark} We can assume without loss of generality that $\E$ has a subset  $\{c_1,\dots,c_n \}$ such that the lattice $\varTheta:=\langle c_1,\dots,c_n\rangle_\Z$ has full rank in $\R^n$. To show this, suppose first that $b\geqslant n$. 
	Take a matrix $R\in\Z^{n\times n}$ whose columns are distinct elements $\{c_1,\dots,c_n \}$ of~$\E$. 
	Consider the integer matrix $N(t) := t \bb- R$, where $t \in \N$ is chosen so that $t\bb\in\Z^{n\times n}$. 
	Its determinant is $\det{(N(t))} = \det{(\bb)}\,\det(t\cdot Id-AR) = \det{(\bb)}\,\chi_{AR}(t)$, where  $\chi_{AR}(t) \in \Q[t]$ is the characteristic polynomial of $AR$. 
	For all but finitely many $t \in \N$, it holds that $\chi_{AR}(t) \ne 0$. 
	Hence, we can choose $t$ in a way that the column vectors $\{\widetilde{c_1},\dots,\widetilde{c_n}\}$ of $N(t)$ are linearly independent, and hence they span a full rank integer lattice. Note that $c_j-\widetilde{c_j}\in\bb\Z^n$ for $j=1,\dots,n$, so we can replace each $c_j$ by $\widetilde{c_j}$, and this produces a new residue set $\widetilde{\E}$ with the required property. 
	
	Suppose now that $1<b< n$. Choose $k\geqslant 1$ so that $b^k\geqslant n$. Note that a complete set of residues for $\Z^n[\bb^k]/\bb^k\Z^n[\bb^k]$ is given by $\E+\bb\E+\dots+\bb^{k-1}\E$, so $|\Z^n[\bb^k]/\bb^k\Z^n[\bb^k]|=b^k\geqslant n$. Suppose that we consider the digit system  $(A^k,\D_k)$ with $\D_k$ as in \eqref{defdk}; then $\E_k$ satisfies the assumption of the previous paragraph. Note that from the iterated set equation \eqref{iterset} it follows that $\F(A^k,\D_k)=\F(A,\D)$. Hence, whenever $|\E|< n$ we can work with $(A^k,\D_k)$ instead of $(A,\D)$ and with $\E+\bb\E+\dots+\bb^{k-1}\E$ instead of $\E$.
	\end{remark}
\subsection{Character theory}
In order to prove the multiple tiling theorem, we use some results on the characters of $\K$. For more on the topic of character theory on locally compact abelian groups, we refer the reader to \cite[Chapter 4]{MR1802924}.

\begin{definition}[Character]
	 A character $\chi$ on a locally compact abelian group $G$ is a continuous function $\chi:G\rightarrow \mathbb{S}^1$ such that $\chi(x+ y)=\chi(x)\chi(y)$ for all $x,y\in G$. 
\end{definition}

\begin{lemma}\label{dualproperties}
	The set of all characters of $G$ constitutes a locally compact abelian group (w.r.t.\ the topology induced by the compact-open topology), called the Pontryagin dual of $G$, denoted by $\widehat{G}$. It satisfies the following properties:\begin{enumerate}
		\item The Pontryagin dual of $\widehat{G}$ is isomorphic to $G$.
		\item The Pontryagin dual of the product $G_1\times G_2$ is isomorphic to $\widehat{G}_1\times \widehat{G}_2$ and the characters are of the form $\chi=\chi_1\cdot\chi_2$ with $\chi_1\in\widehat{G_1}$ and $\chi_2\in\widehat{G_2}.$
		\item Given a subgroup $H\subset G$, define the annihilator of $H$ on $G$ as \[\mathop{Ann} (H):=\{\chi\in\widehat{G}\;|\;\chi(H)=1 \}.\]
		Then $\widehat{(G/H)}\simeq\mathop{Ann}(H)$ and $\widehat{H}\simeq\widehat{G}/\mathop{Ann}(H)$.

	\end{enumerate} 
\end{lemma}
\begin{proof}
	See \cite[Chapter 4]{MR1802924}.
\end{proof}

It is well known that the characters on $\R^n$ are given by
\[
\exp_r:\R^n\rightarrow \mathbb{S}^1;\quad x\mapsto\exp(2\pi i\langle x,r\rangle),
\] where $r\in\R^n$, and $\widehat{\R}^n\simeq\R^n$ via the isomorphism $r\mapsto \exp_r$.

%Since $\K=\R^n\times\qbb$, its dual is $\widehat{\K}=\R^n\times\widehat{\qbb}$, hence we just need to study the characters of $\qbb$.
	
	For any $s=\sum_{j=\nu^*(s)}^{\infty} {\bb^*}^{j}s_j\in\qbbdual$ with $s_j\in\E^*$, define the map
	\begin{equation}\label{characters}
	\psi_s:\qbb\rightarrow \mathbb{S}^1;\quad y\mapsto\exp(2\pi i S_s(y)),
	\end{equation}
	where
	\begin{equation}\label{characterfunction}
	S_s(y):=\sum_{j=\nu^*(s)}^{\infty} \langle \{\bb^jy\}_{\bb},s_j\rangle.
	\end{equation}
	%We will show that $\widehat{\qbb}\simeq\qbbdual$ via the isomorphism $s\mapsto\chi_s$. 
	The map $S_s$ is well defined because $\{\bb^jy\}_{\bb}=0$ for all but finitely many indices $j$. We show next that this map is indeed a character.

\begin{proposition}\label{multiplicativity}
	For every $s\in\qbbdual$, the map $\psi_s$ defined in (\ref{characters}) is continuous and multiplicative, that is, $\psi_s(y+y')=\psi_s(y)\psi_s(y')$.
\end{proposition}
	\begin{proof}
	Fix $s=\sum_{j=\nu^*(s)}^{\infty} {\bb^*}^{j}s_j\in\qbbdual$ with $s_j\in\E^*$. For the multiplicativity, it suffices to show that $S_s(y+y')=S_s(y) + S_s(y')\mod \Z$. We prove first the following claim: Given $\omega,\omega'\in\qbb$, it holds that
		\begin{equation}\label{eq:difinlambda}
		\{\omega \}_{\bb}+\{\omega' \}_{\bb}-\{\omega+\omega' \}_{\bb}\in\Lambda,
		\end{equation}
		where $\Lambda$ is the lattice in $\R^n$ satisfying (\ref{eq:duallattice}). By definition of the $\bb$-adic fractional part, $\{\omega \}_{\bb}+\{\omega' \}_{\bb}-\{\omega+\omega' \}_{\bb}\in\Z^n[A].$ Also,
		\begin{equation*}
		\begin{split}
			\{\omega \}_{\bb}+\{\omega' \}_{\bb}-\{\omega+\omega' \}_{\bb}&=(\{\omega \}_{\bb}-\omega)+(\{\omega' \}_{\bb}-\omega')+(\omega+\omega'-\{\omega+\omega' \}_{\bb})\\
			&=-\lfloor \omega \rfloor_{\bb}-\lfloor \omega' \rfloor_{\bb}+\lfloor \omega+\omega' \rfloor_{\bb}\in\Z^n[A^{-1}].
		\end{split} 
		\end{equation*}
		Since $\Z^n[A]\cap\Z^n[A^{-1}]\subset\Lambda$ by the definition of $\Lambda$, this yields the claim. %Next, assume that $y,y'$ have a finite adic expansion, that is, $y_j=y'_j=0$ for $j\leqslant l$. 
		Now, for any $y,y'$ we have
		\begin{equation}\label{eq:sumSSS}
		S_s(y)+S_s(y')-S_s(y+y')=\sum_{j=\nu^*(s)}^{\infty} \langle \{\bb^jy\}_{\bb}+\{\bb^jy'\}_{\bb}-\{\bb^j(y+y')\}_{\bb},s_j\rangle
		\end{equation}
		where $\{\bb^jy\}_{\bb}+\{\bb^jy'\}_{\bb}-\{\bb^j(y+y')\}_{\bb}\in\Lambda$ by \eqref{eq:difinlambda}. The summands in \eqref{eq:sumSSS} are nonzero only for a finite number of $j$'s. For every index $j$, we have $s_j\in\E^*\subset\Gamma\subset \Lambda^*$, thus by definition of dual lattice, $$\langle \{\bb^jy\}_{\bb}+\{\bb^jy'\}_{\bb}-\{\bb^j(y+y')\}_{\bb},s_j\rangle\in\Z$$
	and the multiplicativity of $\psi_s$ is established.
		
		 For the continuity, let $y,y'\in\qbb$ such that ${\bf d}_\bb(y,y')\leqslant b^{\nu^*(s)}$. Then, for every $j\geqslant \nu^*(s)$, it holds that ${\bf d}_\bb(\bb^jy,\bb^jy')\leqslant 1$, and so $\bb^j(y-y')\in\zbb$  which implies $\{\bb^j(y-y')\}_\bb=0$. Then $S_s(y-y')=0$ and hence, by multiplicativity, $\psi_s(y)=\psi_s(y')$. Thus $\psi_s$ is locally constant and, hence, continuous.
	\end{proof}
	 We will show that the Pontryagin dual of $\K$ is isomorphic to $\R^n\times\qbbdual.$
	To do so, we prove some lemmas first.
	\begin{lemma}\label{equivdef}
		Let $y=\sum_{k=\nu(y)}^{\infty}\bb^ky_k\in\qbb$ with $y_k\in\E$ and $s=\sum_{j=\nu^*(s)}^{\infty} {\bb^*}^{j}s_j\in\qbbdual$ with $s_j\in\E^*$ be given. Then
		\[
		S_s(y)=\sum_{j=\nu^*(s)}^{\infty} \langle \{\bb^jy\}_{\bb},s_j\rangle=\sum_{k=\nu(y)}^{\infty}\langle y_k,\{{\bb^*}^k s\}^*_{\bb} \rangle.
		\]
	\end{lemma}
\begin{proof}From direct calculation, we obtain
	\[\begin{split}
	S_s(y)&=\sum_{j=\nu^*(s)}^{\infty} \langle \{\bb^jy\}_{\bb},s_j\rangle=\sum_{j=\nu^*(s)}^{\infty} \bigg\langle \sum_{k=\nu(y)}^{-j-1}\bb^{j+k}y_k,s_j\bigg\rangle\\
	&=\sum_{j=\nu^*(s)}^{\infty}\sum_{k=\nu(y)}^{-j-1} \langle y_k,\bb^{*j+k}s_j\rangle=\sum_{k=\nu(y)}^{\infty}\sum_{j=\nu^*(s)}^{-k-1} \langle y_k,\bb^{*j+k}s_j\rangle\\
	&=\sum_{k=\nu(y)}^{\infty}\langle y_k,\sum_{j=\nu^*(s)}^{-k-1} \bb^{*j+k}s_j\rangle=\sum_{k=\nu(y)}^{\infty}\langle y_k,\{{\bb^*}^k s\}^*_{\bb} \rangle. \qedhere
	\end{split}
	\]
\end{proof}

Our next step is to establish a Pontryagin duality between $\qbb$ and $\qbbdual$. For that purpose, we express each of these groups in terms of a direct limit of a projective limit.  For more on the topic we refer the reader to \cite{MR1760253}. 
	For each $k\in \N$, consider the quotients
	$$
	%\E_k:=\Z^n[A^{-1} ]/A^{-k}\Z^n[A^{-1} ].
	\E_k:=\Z^n[\bb ]/\bb^{k}\Z^n[\bb].
	$$ 
	Clearly,
	$\E_k\subset\E_{k+1}$ for every $k$, so we can define the canonical projections  
	$$
	\pi_k:\E_{k+1}\rightarrow \E_k,\quad x\mapsto x\mod \bb^{k}. 
	$$Therefore, we have a projective system
\[
\dots \longrightarrow \E_{k+1}\stackrel{\pi_{k}}{\longrightarrow}\E_k\stackrel{\pi_{k-1}}{\longrightarrow}\E_{k-1}\longrightarrow \dots \stackrel{\pi_{1}}{\longrightarrow}\E_1\stackrel{\pi_{0}}{\longrightarrow}\E_0,
\]
which entitles the existence of the projective limit $$\varprojlim_{k\in\N}
			\mathcal{E}_k=\{(M_k)_{k\in \N}\;|\;M_k\in\E_k\mbox{ and }
			\pi_k(M_{k+1})=M_k \mbox{ for every }k\},$$		
			and it holds that \begin{equation}\label{projinjlim}\qbb\simeq\varinjlim_{j\in\N}\varprojlim_{k\in\N}\bb^{-j}
			\mathcal{E}_k.\end{equation} Analogously, for $k\in \N$ consider $$\E^*_k:=\Gamma[\bb^*]/\bb^{*k}\Gamma[\bb^*].$$ Then $$\qbbdual\simeq\varinjlim_{j\in\N}\varprojlim_{k\in\N}{\bb^*}^{-j}\E^*_k.$$

	\begin{proposition}\label{dualofK}	The characters of $\K$ are of the form
	\[\chi_{r,s}:\K\rightarrow S^1,\quad\chi_{r,s}(x,y)=\exp_r(x)\psi_s(y)=\exp(2\pi i\langle x,r\rangle)\exp(2\pi i S_s(y)),
	\] for $r\in\R^n$, $s\in\qbbdual$, with $S_s(y)$ as in (\ref{characterfunction}).
	%\[S_s(y):=\sum_{j=\nu^*(s)}^{\infty} \langle \{A^jy\}_{\bb},s_j\rangle.
	%\]
	Moreover, there is a group isomorphism %between $\R^n\times\qbbdual$ and $\{\chi_{(r,s)}\;|\;(r,s)\in\R^n\times\qbbdual\}$ 
	given by 
 $(r,s)\mapsto \chi_{r,s}$.% is an injective group morphism.
\end{proposition}
		\begin{proof}
		In view of Lemma \ref{dualproperties} we have the isomorphism $\widehat{\K}\simeq\widehat{\R^n}\times\widehat{\qbb}$. It is known that there is an isomorphism $\R^n\simeq\widehat{\R^n}$ given by $r\mapsto \exp_r$. Consider the map $\qbbdual\rightarrow\widehat{\qbb}$, $s\mapsto \psi_s$; we show that this map is a group isomorphism. %where $\chi_s$ is the character of $\qbb$ defined in \ref{characters}. 
		We first prove that the group operations are compatible on both sets, that is, $\psi_{s+s'}(y)=\psi_s(y)\psi_{s'}(y)$ for every $y=\sum_{j=\nu(y)}^{\infty}\bb^jy_j\in\qbb$ with $y_j\in\E$. It is enough to show that $S_{s+s'}(y)=S_s(y)+S_{s'}(y)\mod\Z$. Applying Lemma \ref{equivdef}, we get %we have that $S_s(y)=\sum_{j=\nu(y)}^{\infty}\langle y_j,\{{\bb^*}^j s\}^*_{\bb} \rangle$, hence
		\[
		S_s(y)+S_{s'}(y)-S_{s+s'}(y)=\sum_{j=\nu(y)}^{\infty}\langle y_j,\{{\bb^*}^j s\}^*_{\bb}+\{{\bb^*}^j s'\}^*_{\bb} -\{{\bb^*}^j (s+s')\}^*_{\bb}\rangle.
		\]
		Proceeding in analogy to the proof of Proposition \ref{multiplicativity} and using the definition of $\bb^*$-adic fractional part, we can see that, for every index $j\geqslant\nu(y)$,
		\[\{{\bb^*}^j s\}^*_{\bb}+\{{\bb^*}^j s'\}^*_{\bb} -\{{\bb^*}^j (s+s')\}^*_{\bb}\subset \Gamma[\bb^*]\cap\Gamma[{\bb^*}^{-1}]\subset\Lambda^*\subset\Z^n,
		\]
		with $\Lambda^*$ as in \eqref{eq:duallattice}. Since $y_j\in\E\subset\Z^n$ for every $j\geqslant\nu(y)$ and is finite only for a finite number of indices, this yields the first part of the proof.
		
		Next, we show the injectivity. In view of the first part of the proof, it suffices to show that $\psi_s\neq1$ for $s\neq0$. Let $s=\sum_{j=\nu^*(s)}^{\infty}\bb^{*j}s_j\in\qbbdual\setminus\{0\}$ with $s_j\in\E^*$, and consider a point of the form $\bb^{-l}c\in\qbb$ for $0\neq c\in\E$ and $l\in\N$. Note that $\{\bb^{j-l}c \}_\bb=0$ whenever $j\geqslant l$. Therefore
		%, applying again Lemma \ref{equivdef} 
		we get
		\[S_s(\bb^{-l}c)=\sum_{j=\nu^*(s)}^{l-1}\langle \bb^{j-l}c,s_j\rangle=\sum_{j=\nu^*(s)}^{l-1}\langle c,{\bb^*}^{j-l}s_j\rangle=\bigg\langle c,\sum_{j=\nu^*(s)}^{l-1}{\bb^*}^{j-l}s_j\bigg\rangle.
		\]
		Suppose $\psi_s=1$, then $S_s(\bb^{-l}c)\in\Z$ for every $l\in\N$ and every $c\in\E$. Recall that by Remark~\ref{spanlatticeremark} we may assume w.l.o.g.\ that $\E$ has a subset $\{c_1,\dots,c_n \}$ such that $\varTheta:=\langle c_1,\dots,c_n\rangle_\Z$ is a full rank lattice in $\R^n$. Thus $S_s(\bb^{-l}\varTheta)\subset\Z$ and hence, by the definition of a dual lattice, $\sum_{j=\nu^*(s)}^{l-1}{\bb^*}^{j-l}s_j\in\varTheta^*$ holds for all $l\in\N$. %		
%		This implies that $\{\bb^{*-l}s\}^*_{\bb}\in\varTheta^*$ for every $l\geqslant 1$. Suppose $s\neq 0$, then $s_{\nu^*(s)}\in\Gamma\setminus B^*\Gamma$ because the coefficients live in $\E^*\subset\Gamma$ which is defined to be a residue set for $\Gamma[\bb^*]/\bb^*\Gamma[\bb^*]$. 
Since $s\neq 0$, we have that
		\[\sum_{j=\nu^*(s)}^{l-1}{\bb^*}^{j-l}s_j\in (\bb^{*\nu^*(s)-l}\Gamma+\dots+\bb^{*-1}\Gamma)\setminus (\bb^{*\nu^*(s)-l+1}\Gamma+\dots+\bb^{*-1}\Gamma).
		\]
		%Define inductively $\Lambda_0:=\Gamma$ and $\Lambda_{l+1}$ the lattice spanned by $\Lambda_l$ and $\{\bb^{*-l}s\}^*_{\bb}$.  
		Then $(\bb^{*\nu^*(s)-l}\Gamma+\dots+\bb^{*-1}\Gamma)_{l\geqslant 1}$ is a strictly nested sequence of rational lattices in $\R^n$. Given $l$, consider the entries of all the vectors of $\bb^{*l-\nu^*(s)}\Gamma+\dots+\bb^{*}\Gamma$ expressed as irreducible fractions, and define $m_l$ to be the maximum of the denominators of these fractions. It is clear that $m_l\leqslant m_{l+1}$. However, note that $(m_l)_{l\geqslant 1}$ does not stabilize: this would imply that, for some lattice $\Lambda_l$, there are infinitely many steps in which we can add a point and form a strictly larger lattice, while not increasing the bound on the denominators, which is not possible.	This means that $\varTheta^*$ contains points with entries having arbitrarily large denominators, which is a contradiction.

		For the surjectivity, consider a character $\psi\in\widehat{\qbb}$. By classical arguments following \cite[p. 139]{MR1802924}, we obtain from \eqref{projinjlim} that
		\begin{equation}\label{dualbadicint}
		\widehat{\qbb}\simeq \varprojlim_{j\in\N}\varinjlim_{k\in\N}\widehat{\bb^{-j}\E_k}.
		\end{equation}
		
			Since $\bb^{-j}\E_k=\bb^{-j}\Z^n[\bb]/\bb^{k-j}\Z^n[\bb]$ is a finite group of cardinality $b^k$, so is its dual. Consider the group
		%\[\{\chi_z\in\widehat{\qbb}\;|\;z=\sum_{l=j-k}^{j-1}\bb^{*l}z_j\mod\bb^{*j}\zbb \}.
		%\]
		\[G_{j,k}:=\bb^{*j-k}\Gamma[\bb^*]/\bb^{*j}\Gamma[\bb^*],
		\]
		and note that $|G_{j,k}|=b^k$ because $\Gamma$ is a full rank lattice (and hence, isomorphic to $\Z^n$). Given $s\in G_{j,k}$, we can regard it as an element in $\bb^{*j-k}\E^*+\dots+\bb^{*j-1}\E^*$ and consider the character $\psi_s$ as in \eqref{characters}. Note that $\psi_s\in\mathop{Ann}(\bb^{k-j}\Z^n[\bb])$ %because the $\bb^*$-adic coefficients of $s$ satisfy $s_l=0$ for $l<j-k$
		(see Lemma \ref{dualproperties}); hence, $\psi_s$ is a character of ${\bb^{-j}\E_k}$. Also, for $s\neq0$ in $G_{j,k}$, there is $y\in{\bb^{-j}\E_k}$ such that $S_s(y)\neq0$: in fact, there exists $l$ with $j-k\leqslant l\leqslant j-1$ with $s_l\neq0$; since $\E$ spans a full rank lattice in $\R^n$, find $c\in\E$ such that $\langle c,s_l \rangle \neq 0$ and take $y=\bb^{-l}c$. Hence all the characters $\psi_s$ for $s\in G_{j,k}$ are distinct over $\bb^{-j}\E_k$, and since $|\bb^{-j}\E_k|=b^k=|G_{j,k}|$, this implies that $\widehat{\bb^{-j}\E_k}\simeq G_{j,k}$, so  \eqref{dualbadicint} yields
		\[
		\widehat{\qbb}\simeq \varprojlim_{j\in\N}\varinjlim_{k\in\N}G_{j,k}
		%\bb^{*j-k}\Gamma[\bb^*]/\bb^{*j}\Gamma[\bb^*]
		%&=\varprojlim_{j\in\N}\bigcup_{k\in\N}\bb^{*j-k}\E^*+\dots+\bb^{j-1}\E^*\mod \bb^{*j}\Z_{\bb^*}.
		\] It is not hard to establish an isomorphism $G_{j,k}\simeq \bb^{*-j}\E^*_k$,		
	% $\widehat{\bb^{-j}\E_k}\simeq G_{j,k}=\bb^{*j-k}\Gamma[\bb^*]/\bb^{*j}\Gamma[\bb^*]$. If we take a character in $\widehat{\bb^{-j}\E_k}$, then it is of the form $\chi_{s(j,k)}$, where $s(j,k)=\bb^{*j-k}s_{j-k}+\dots+\bb^{*j-1}s_{j-1}\in\qbbdual$ with $s_i\in\E^*$, 
	and so $$\psi\in\varprojlim_{j\in\N}\varinjlim_{k\in\N}\{\psi_s\;|\;s\in G_{j,k}\}\simeq\varprojlim_{j\in\N}\varinjlim_{k\in\N}\bb^{*-j}\E^*_k\simeq\qbbdual,$$ therefore every character $\psi$ is of the form $\psi=\psi_s$ for some $s\in\qbbdual$.	
%		
%		
%		
%		Now, $\sum_{j=\nu^*(s)}^{l-1}{\bb^*}^{l-j}s_j$ is a vector in $\Q^n$; if we express its entries as reduced fractions, their denominators tend to infinity as $l\to\infty$, because $\bb^*$  is contractive (recall that all eigenvalues of $\bb^*$ are less than $1$ in modulus). This means that for each $N>0$ the lattice $\varTheta^*$ a vector having an entry which is a (reduced) fraction having a denominator that is larger than $N$. This is a contradiction. \commL{modify}
	\end{proof}	
	\subsection{Multiple tiling theorem}

In this final section, we will prove that rational self-affine tiles give a multiple tiling of the representation space. We will make use of the character theory we have developed in the previous section.

Recall that the set $\varphi(\Z^n[A])$ is a lattice in $\K$ by Lemma \ref{lattice}, hence the torus $\mathbb{T}:=\K/\varphi(\Z^n[A])$ is well defined and compact. We endow it with the normalized quotient measure $\bar{\mu}$, which is the Haar measure on $\mathbb{T}$. Denote the multiplication by $A$ on $\mathbb{T}$ as%\[\tau_A:\mathbb{T}\rightarrow\mathbb{T},\quad
$$\tau_A:\mathbb{T}\rightarrow\mathbb{T};\quad(x,y)\mapsto A(x,y)\mod\varphi(\Z^n[A]),$$
	which is well defined because $\varphi(\Z^n[A])$ is $A$-invariant. 	
%	Since  is a fundamental domain for $\varphi(\Z^n[A])$ we can identify $\K/\varphi(\Z^n[A])$ with $\mathbb{T}$, and the characters on $\K/\varphi(\Z^n[A])$ with the characters on $\mathbb{T}$. The multiplication by $A$ on $\K/\varphi(\Z^n[A])$ can be expressed in terms of class representatives on the fundamental domain by
%	\[
%	\tau_A:\mathbb{T}\rightarrow\mathbb{T},\quad
%	(x,y)\mapsto (\{Ax\},\lfloor Ay\rfloor_\bb).\]	
	We will prove that this map is ergodic by using the following lemma.
	
		\begin{lemma}\label{ergodicitylemma}
		If $G$ is a compact abelian group with normalized Haar measure and $\tau:G\rightarrow G$ is a surjective continuous endomorphism of $G$, then $\tau$ is ergodic if and only if the trivial character $\chi =1$ is the only character of $G$ that satisfies $\chi \circ \tau^k=\chi$ for some $k\geqslant1$.
	\end{lemma}
	\begin{proof}
		See \cite[Theorem 1.10.1]{MR648108}.
	\end{proof}
	We are now in a position to prove the following result.
	\begin{lemma}\label{Aergodic}
	The map $\tau_A$ is ergodic.
	\end{lemma}
\begin{proof} Because $A$ is an invertible matrix, $\tau_A$ is a continuous surjective homomorphism.	We first prove that it is measure preserving. Note that $\varphi(\Z^n[A])$ is a sublattice of index $a$ of $A^{-1}\varphi(\Z^n[A])$ in $\K$. This implies that, for any measurable set $E\subset\mathbb{T}$,
	\[	\begin{split}
	\bar{\mu}(\tau_A^{-1}(E))&=\bar{\mu}(\{(x,y)\in\mathbb{T}\;|\;A(x,y)\mod\varphi(\Z^n[A])\in E\})\\
	&=\bar{\mu}(\{(x,y)\in\mathbb{T}\;|\;(x,y)\mod A^{-1}\varphi(\Z^n[A])\in A^{-1}E\})\\
	&=a\,\bar{\mu}(\{(x,y)\in\mathbb{T}\;|\;(x,y)\mod \varphi(\Z^n[A])\in A^{-1}E\})\\
	&=a\,\bar{\mu}(A^{-1}E)=\bar{\mu}(E).\\
	\end{split}
	\]
	
	By Proposition \ref{dualofK}, the characters of $\K$ are of the form $\chi_{r,s}$ for $r\in\R^n$, $s\in\qbbdual$. Since $\mathbb{T}=\K/\varphi(\Z^n[A])$, by Lemma~\ref{dualproperties} there is an isomorphism $\widehat{\mathbb{T}}\simeq \mathop{Ann}(\varphi(\Z^n[A]))$; this means that a character of $\mathbb{T}$ is of the form $\chi_{r,s}$ and satisfies $\chi_{r,s}(\varphi(z))=1$ for every $z\in\Z^n[A]$. Suppose that a character $\chi_{r,s}\in\widehat{\mathbb{T}}$ satisfies $\chi_{r,s}\circ\tau^k_A=\chi_{r,s}$ for some $k\geqslant 1$. In view of Lemma \ref{ergodicitylemma}, it suffices to show that this implies that $\chi_{r,s}$ is constantly equal to $1$, {\it i.e.}, that $\chi_{r,s}=\chi_{0,0}$. We have that, for every $(x,y)\in\mathbb{T}$, 
	\[\chi_{r,s}\circ\tau^k_A(x,y)=\chi_{r,s}(x,y),
	\]
	and, hence,
\[\exp(2\pi i(\langle r,A^kx\rangle + S_s(A^ky)))=\exp(2\pi i(\langle r,x\rangle + S_s(y)))
\]	with $S_s$ defined in \eqref{characterfunction}. 
%	\[\chi_r(\{A^kx\})\chi_s(\lfloor A^ky\rfloor_\bb)=\chi_r(x)\chi_s(y)
%	\]
	Letting $y=0$ implies
%	\[\exp(2\pi i \langle r,A^kx\rangle)=\exp(2\pi i \langle r,x \rangle)
%	\]
%	and hence
	\[\langle r, A^kx\rangle=\langle r, x\rangle \mod \Z
	\]
	and so, for every $x\in[0,1]^n$,
	\[\langle {A^*}^kr-r,x\rangle \in\Z,
	\]
	which can only be true if ${A^*}^kr-r=0$. Suppose $r\neq0$; this implies that ${A^*}^k$ has $1$ as an eigenvalue, and so therefore $A$ has $1$ as an eigenvalue, which contradicts the fact that $A$ is expanding. Hence, $r=0$ and the character has to be of the form $\chi_{0,s}$.
	
Next, we prove that $\chi_{0,s}\in\widehat{\mathbb{T}}$ implies $s=0$. 
For every $z\in\Z^n[A]$	we have $1=\chi_{0,s}(\varphi(z))=\exp_0(z)\psi_s(z)=\psi_s(z)$. Consider points of the form $z=A^lc=\bb^{-l}c$ with $l\geqslant0$, $c\in\E$. Then $S_s(\bb^{-l}c)=\langle c,\{\bb^{*-l}s\}^*_{\bb}\rangle\in\Z$. Recall that $\E$ contains a subset that spans a full rank integer lattice $\varTheta$ in $\R^n$. By the definition of dual lattice, this implies that $\{\bb^{*-l}s\}^*_{\bb}\in\varTheta^*$ for every $l\geqslant 0$. Suppose $s\neq 0$, then the leading $\bb^*$-adic coefficient $s_{\nu^*(s)}$ of $s$ satisfies $s_{\nu^*(s)}\in\Gamma\setminus B^*\Gamma$, because the $\bb^*$-adic coefficients live in $\E^*\subset\Gamma$ which is defined to be a complete set of residue class representatives of $\Gamma[\bb^*]/\bb^*\Gamma[\bb^*]$. Moreover, for $l \ge \max\{\nu^*(s),0\}$ we have that
	\[\{\bb^{*-l}s\}^*_{\bb}\in (\bb^{*\nu^*(s)-l}\Gamma+\dots+\bb^{*-1}\Gamma+\Gamma)\setminus (\bb^{*\nu^*(s)-l+1}\Gamma+\dots+\bb^{*-1}\Gamma+\Gamma).
	\]
Let $\Lambda_{\max\{\nu^*(s),0\}}$ be the lattice in $\R^n$ spanned by $\{\bb^{*-\max\{\nu^*(s),0\}}s\}^*_{\bb}$ and define $\Lambda_{l+1}$ inductively as the lattice spanned by $\Lambda_l$ and $\{\bb^{*-l-1}s\}^*_{\bb}$.  Then $(\Lambda_l)_{l\geqslant \max\{\nu^*(s),0\}}$ is a strictly nested sequence of rational lattices in $\R^n$, each of which is contained in $\varTheta^*$. Given $l$, consider the entries of all the vectors of $\Lambda_l$ expressed as irreducible fractions, and define $m_l$ to be the maximum of the denominators of these fractions. It is clear that $m_l\leqslant m_{l+1}$. However, note that $(m_l)_{l\geqslant \max\{\nu^*(s),0\}}$ does not stabilize: if it would stabilize, this would imply that, for some lattice $\Lambda_l$, there are infinitely many steps in which we can add a point and form a strictly larger rational lattice, while not increasing the bound on the denominators, which is not possible.	
%	
%	and $\Gamma+A^*\Gamma+\dots+A^{*l-1}\Gamma$ is a strictly nested sequence of lattices for every $l$. Since $\Gamma$ is a lattice in $\Q^n$, its elements are rational vectors whose entries (expressed as irreducible fractions) have bounded denominators. When we consider $\Gamma+A^*\Gamma+\dots+A^{*l-1}\Gamma$, the lattice gets strictly finer.
%	
%	We are adding points to $\varTheta$ that get arbitrarily large denominators (EXPAND)
%	
Hence the sequence of denominators $(m_l)_{l\geqslant \max\{\nu^*(s),0\}}$ tends to infinity. This contradicts the uniform discreteness of the lattice $\varTheta^*$. Thus $s=0$ and, hence, $\chi_{r,s}=\chi_{0,0}=1$ as desired.
%	\[\chi_s(\lfloor A^ky\rfloor_\bb)=\chi_s(y)
%	\]
%	for every $y\in\sbb$. Hence
%	\[S_z(y)-S_z(\lfloor A^ky\rfloor_\bb)\in\Z.
%	\]
%	
%	Let $y=y_0\in\E\subset\Z^n$, then, since $k\geqslant 1$ we have $\lfloor A^ky\rfloor_\bb=0$. Suppose $z\neq 0$; then
%	\[\sum_{j=1}^{\nu^*(z)}\langle \{A^j y_0 \}_{\bb},z_j\rangle=\sum_{j=1}^{\nu^*(z)}\langle A^j y_0 ,z_j\rangle =\langle y_0,z \rangle \in\Z
%	\]
%	and $z_{\nu^*(z)}\neq 0$. We need to show that there is a choice of $y_0 \in \E$ such that $\langle y_0,z \rangle \notin\Z$.
%	\commL{ask this}
\end{proof}

	We arrive at our final result. The proof is based on the one of \cite[Theorem~1.1]{MR1395075}. For completeness, we include it here. We remind the reader of the definition of multiple tiling given in Definition~\ref{deftiling}.
	
\begin{theorem}\label{multipletilingthm}
	Let $\F=\F(A,\D)$ be a rational self-affine tile. Then $\F+\varphi(\Z^n[A])$ is a multiple tiling of $\K$.
\end{theorem}

\begin{proof}
	Let $\bar\mu$ be the normalized Haar measure on the torus $\mathbb{T}=\K/\varphi(\Z^n[A])$.  Consider the canonical projection $\pi:\K\rightarrow\mathbb{T}$, and define the function $\varPhi:\mathbb{T}\rightarrow\N$ as
	\[ \varPhi(x,y):= | \pi^{-1}(x,y)\cap \F|,
	\]
	meaning $\varPhi$ counts the points on $\F$ that are congruent to $(x,y)$ modulo $\varphi(\Z^n[A])$. Since $\F$ is compact, $\varPhi$ is finite everywhere, hence it is well defined. Also, $\varPhi$ is positive in a set of positive measure since $\mu(\F)>0$.
	If we prove that $\varPhi(x,y)$ is equal almost everywhere to some $k\in\N$, this implies the statement of the theorem: it means that almost every point of $\K$ gets covered by exactly $k$ translates of $\F$ when translating via the set $\varphi(\Z^n[A])$.
	Note that $\varPhi$ is constant almost everywhere if and only if there exists $k$ such that every $S\subset \mathbb{T}$ satisfies
	\begin{equation}\label{eqinvariant}
	\int_S\varPhi(x,y)\,d\bar\mu(x,y)=k\,\bar\mu(S).
	\end{equation}
	
	To obtain this, we show first that $\varPhi$ satisfies a.e.\
	\begin{equation}\label{eqidentity}
	\varPhi(x,y)=\frac1a\sum_{(x',y')\in \tau_A^{-1}(x,y)}\varPhi(x',y'),
	\end{equation} where $\tau_A$ is the multiplication by $A$ on the torus. 
	
	Using the essential disjointness of the set equation \eqref{seteq} of $\F$ (see \eqref{disjoint}) and the fact that $\varphi(\D) \subset\varphi(\Z^n[A])$, we have that a.e.\
\begin{equation}\label{multtil1}
\begin{split}
|\pi^{-1}(x,y)\cap A\F| &=\Big|\pi^{-1}(x,y)\cap\Big( \bigcup_{d\in\D}(\F+\varphi(d))\Big) \Big|
%\\
%&=\Big|\bigcup_{d\in\D}(\pi^{-1}(x,y)-\varphi(d))\cap \F\Big| \\
%&=\sum_{d\in\D}|\pi^{-1}(x,y)\cap \F|
=a\,\varPhi(x,y).
\end{split}
\end{equation}
We have
\[
A^{-1}\pi^{-1}(x,y) = A^{-1}((x,y) + \varphi(\Z^n[A])) 
%= \tau_A^{-1}(x,y) + \varphi(\Z^n[A]) 
= \pi^{-1}( \tau_A^{-1}(x,y))
\]
and, hence, $\pi^{-1}(x,y)=A \pi^{-1}( \tau_A^{-1}(x,y))$. Because $A$ is a bijection on $\R^n$ this implies that
\begin{equation}\label{multtil6}
\begin{split}
|\pi^{-1}(x,y)\cap A\F| &= |(A \pi^{-1}( \tau_A^{-1}(x,y))) \cap A\F| 
\\
&= |\pi^{-1}( \tau_A^{-1}(x,y)) \cap \F|  \\
&=\sum_{(x',y')\in \tau_A^{-1}(x,y)}|\pi^{-1}(x',y')\cap \F|
\\
&=\sum_{(x',y')\in \tau_A^{-1}(x,y)} \varPhi(x',y').
\end{split}
\end{equation}
%Because the index of the sublattice $A\varphi(\Z^n[A])$ of $\varphi(\Z^n[A])$ is $a$, we have $|\tau_A^{-1}(x,y)|=a$, and thus
%\begin{equation}\label{multtil6}
%|\pi^{-1}(x,y)\cap A\F| =\frac1a\sum_{(x',y')\in \tau_A^{-1}(x,y)}|\pi^{-1}(\tau_A(x',y'))\cap A\F|.
%\end{equation}
%Also, for any $(u,v)\in\K$, it holds that $|((u,v)+\varphi(\Z^n[A]))\cap A\F| =a | ((u,v)+A\varphi(\Z^n[A]))\cap A\F |$, and hence, for each $(x',y')\in \tau_A^{-1}(x,y)$,
%\begin{equation}\label{multtil4}
%|\pi^{-1}(\tau_A(x',y'))\cap A\F|=a\,|A\pi^{-1}(x',y')\cap A\F|=a\,|\pi^{-1}(x',y')\cap \F|=a\,\varPhi(x',y').
%\end{equation}
%Thus from \eqref{multtil6} and \eqref{multtil4} follows that
%
%\begin{equation}\label{multtil2}
%\begin{split}
%|\pi^{-1}(x,y)\cap A\F| =\sum_{(x',y')\in \tau_A^{-1}(x,y)}\varPhi(x',y').
%\end{split}
%\end{equation}	
%
%Combining (\ref{multtil1}) and (\ref{multtil2}) 
Combining (\ref{multtil1}) and (\ref{multtil6}) yields (\ref{eqidentity}). Applying (\ref{eqidentity}) %and doing  
we get
	\begin{equation}\label{multtil5}
	\begin{split}
\int_S\varPhi(x,y)d\bar\mu(x,y)&=\int_S \frac1a\sum_{(x',y')\in \tau_A^{-1}(x,y)}\varPhi(x',y')\,d\bar{\mu}(x,y)\\
%&=\int_{\tau_A^{-1} (S)} \frac1a\sum_{\tau_A(x',y')=\tau_A(x,y)}\varPhi(x',y')\,a\,d\bar{\mu}(x,y)\\
%&=\int_{\tau_A^{-1} (S)} \varPhi(x,y)\,d\bar{\mu}(x,y).\\
%\
\end{split}
\end{equation}

Let $F\subset\K$ be a fundamental domain for the lattice $\varphi(\Z^n[A])$. Because $\varphi(\Z^n[A])$ is a sublattice of index $a$ of $A^{-1}\varphi(\Z^n[A])$ on $\K$, this implies that the collection $\{\pi(A^{-1}F+z)\,|\,z\in A^{-1} \varphi(\Z^n[A])\}$ has exactly $a$ elements. Call these sets $M_1,\dots,M_a$. These sets are pairwise essentially disjoint and their union is the torus $\mathbb{T}$. For each $j=1,\dots,a$, the restricted map %$\tau_A{\big|}_{M_j}:M_j\rightarrow \mathbb{T}$
 ${\tau_A}_{\big|M_j}:M_j\rightarrow \mathbb{T}$ is a bijection. It holds that $\tau^{-1}_A(x,y)=\{ {\tau^{-1}_A}_{\big|M_j}(x,y)\,|\,j=1,\dots,a \}.$ Hence, we can rewrite \eqref{multtil5} as
 %the map $\tau_A$ is $a$ to $1$. Also $\tau_A$ is measure preserving, hence
%\begin{equation}
%\int_S \frac1a\sum_{(x',y')\in \tau_A^{-1}(x,y)}\varPhi(x',y')\,d\bar{\mu}(x,y)=.
%\end{equation}
\begin{equation}\label{multtil7}
\begin{split}
\int_S\varPhi(x,y)d\bar\mu(x,y)&=\int_{S} \frac1a \,\sum_{j=1}^a \varPhi({\tau^{-1}_A}_{\big|M_j}(x,y))\,d\bar{\mu}(x,y)\\
&=\sum_{j=1}^a\int_{S\cap M_j} \frac1a \, \varPhi({\tau^{-1}_A}_{\big|M_j}(x,y))\,d\bar{\mu}(x,y).\\
\end{split}
\end{equation}
 For each $j=1,\dots,a$ we apply the change of variables $(x,y)\mapsto {\tau_A}_{\big|M_j}(x,y)$ in each corresponding integral in \eqref{multtil7}. The Jacobian of each of these maps is $a$, so it follows that
 \begin{equation}\label{multtil8}
 \begin{split}
 \int_S\varPhi(x,y)d\bar\mu(x,y)&=\sum_{j=1}^a\int_{{\tau^{-1}_A}_{\big|M_j}(S\cap M_j)} \frac1a \, \varPhi(x,y)\,a\,d\bar{\mu}(x,y)\\
 &=\int_{\tau_A^{-1} (S)} \varPhi(x,y)\,d\bar{\mu}(x,y).
 \end{split}
 \end{equation}
Because the map $\tau_A$ is ergodic by Lemma \ref{Aergodic}, iterating \eqref{multtil8} $j\geqslant 1$ times and afterwards applying the ergodic theorem (see \cite[Theorem 3.20]{MR2723325}), yields
\begin{equation}
\begin{split}
\int_S\varPhi(x,y)d\bar\mu(x,y)&=\int_{\tau_A^{-j}(S)} \varPhi(x,y)d\bar{\mu}(x,y)\\
&=\int_{\mathbb{T}} \mathbbm{1}_S(\tau_A^j (x,y)) \varPhi(x,y)d\bar{\mu}(x,y)\\
&=\int_{\mathbb{T}} \Big(\frac1N \sum_{j=0}^{N-1}\mathbbm{1}_S(\tau_A^j (x,y))\Big) \varPhi(x,y)d\bar{\mu}(x,y)\\
& \xrightarrow[N \to \infty]{} \bar{\mu}(S)\int_{\mathbb{T}} \varPhi(x,y)d\bar{\mu}(x,y)=k\,\bar{\mu}(S),
\end{split}
\end{equation}
and this concludes the proof.
\end{proof}

\bibliography{bibliografia}
\bibliographystyle{plain}

\end{document}